\documentclass[11pt, reqno]{amsart}
\usepackage{amsmath,amsthm,amscd,amssymb,amsfonts, amsbsy}
\usepackage{latexsym, color, enumerate}
\usepackage{pxfonts}
\usepackage{mathrsfs}
\usepackage{marginnote}
\usepackage{todonotes}
\usepackage[colorlinks=true, pdfstartview=FitV, linkcolor=blue, citecolor=blue, urlcolor=blue]{hyperref}
\usepackage{mathtools}
\mathtoolsset{showonlyrefs}
\usepackage[margin=1 in]{geometry}

\allowdisplaybreaks[4]

\usepackage[numbers,sort]{natbib}

\newtheorem{innercustomgeneric}{\customgenericname}
\providecommand{\customgenericname}{}
\newcommand{\newcustomtheorem}[2]{%
  \newenvironment{#1}[1]
  {%
   \renewcommand\customgenericname{#2}%
   \renewcommand\theinnercustomgeneric{##1}%
   \innercustomgeneric
  }
  {\endinnercustomgeneric}
}

\newcustomtheorem{customthm}{Theorem}

\usepackage{tikz}
\usetikzlibrary{snakes}
\usepackage{multicol}
\usepackage{cancel}

\numberwithin{equation}{section}

\theoremstyle{plain}
\newtheorem{theorem}{Theorem}[section]
\newtheorem{lemma}[theorem]{Lemma}
\newtheorem{corollary}[theorem]{Corollary}
\newtheorem{proposition}[theorem]{Proposition}

\theoremstyle{definition}

\newtheorem{assumption}[theorem]{Assumption}
\newtheorem{remark}[theorem]{Remark}

\newtheorem*{theorem*}{Theorem}
\newtheorem*{question}{Question}

\newcommand{\bC}{\mathbb{C}}

\newcommand{\bR}{\mathbb{R}}

\newcommand\cA{\mathcal{A}}

\newcommand\cL{\mathcal{L}}

\newcommand\cS{\mathcal{S}}

\newcommand\tA{\widetilde{A}}
\newcommand\tr{\operatorname{tr}}

\DeclareMathOperator\supp{supp}
\newcommand{\abs}[1]{\lvert#1\rvert}
\newcommand{\norm}[1]{\left\|#1\right\|}
\newcommand{\inner}[1]{\left \langle#1\right \rangle}

\newcommand{\parenthese}[1]{\left(#1\right)}

\renewcommand{\vec}[1]{\boldsymbol{#1}}

\def\XXint#1#2#3{{\setbox0=\hbox{$#1{#2#3}{\int}$}
		\vcenter{\hbox{$#2#3$}}\kern-.5\wd0}}

\newcommand{\p}{\partial}
\newcommand{\epsi}{\varepsilon}

\begin{document}

\title[Unique continuation]{On the uniqueness of variable coefficient Schr\"odinger equations}

\author[Federico, Li and Yu]{Serena Federico, Zongyuan Li and Xueying Yu}

\address{Serena Federico
\newline \indent Dipartimento di Matematica, Universit\`a di Bologna
\newline \indent Piazza di Porta San Donato 5, 40126, Bologna, Italy}
\email{serena.federico2@unibo.it}

\address{Zongyuan Li
\newline \indent Department of Mathematics, Rutgers University, \indent 
\newline \indent 110 Frelinghuysen Road, Piscataway, NJ 08854-8019, USA}
\email{zongyuan.li@rutgers.edu}


\address{Xueying  Yu
\newline \indent Department of Mathematics, Oregon State University\indent 
\newline \indent  Kidder Hall 368
Corvallis, OR 97331, USA \indent 
\newline \indent 
And 
\newline \indent 
Department of Mathematics, University of Washington, 
\newline \indent  C138 Padelford Hall Box 354350, Seattle, WA 98195, USA}

\email{xueying.yu@oregonstate.edu}


\begin{abstract}
We prove unique continuation properties for linear variable coefficient Schr\"odinger equations with bounded real potentials. Under certain smallness conditions on the leading coefficients, we prove that solutions decaying faster than any cubic exponential rate at two different times must be identically zero. Assuming a transversally anisotropic type condition, we recover the sharp Gaussian (quadratic exponential) rate in the series of works by Escauriaza-Kenig-Ponce-Vega \cite{EKPV_CPDE, EKPV_JEMS, EKPV_Duke}. 
\end{abstract}

\maketitle

\setcounter{tocdepth}{1}
\tableofcontents

\parindent = 10pt     
\parskip = 8pt

\section{Introduction}
In this paper, we consider Schr\"odinger equations with variable coefficients of the form
\begin{equation}\label{eqn-220312-0505}
	\p_t u = i (\cL + V)\,  u,\quad (t,x)\in[0,1]\times \bR^n.
\end{equation}
Here
\begin{equation}\label{eqn-220517-0913}
	\cL = \cL (x) = \sum_{j,k=1}^n\p_k(a_{kj}(x)\p_j),
\end{equation}
with $A(x) = (a_{kj} (x))_{k,j = 1, \cdots ,n}$ being a real symmetric matrix. Throughout this paper, we assume the boundedness and ellipticity 
\begin{equation} \label{eqn-221030-0836}
	\exists\lambda,\Lambda>0, \,\, \forall x,\xi\in\mathbb{R}^n, \quad \lambda|\xi|^2 \leq \sum_{j,k=1}^na_{jk}(x)\xi_j \xi_k\leq \Lambda |\xi|^2.
\end{equation}
Due to the symmetry of $A(x)$, the operator $\cL$ is self-adjoint, that is $\cL=\cL^*$. 
In our main results, the potential is assumed to be real valued and bounded, that is, $V\in L^\infty(\mathbb{R}^n, \mathbb{R})$. 
Moreover, we will often omit the dependence on $x$ and simply write $A, a_{jk}$, and $V$, in place of $A(x),a_{jk}(x)$, and $V(x)$.

The aim of this work is to find the (minimal) assumptions on the variable coefficients $a_{kj}$ such that
\begin{center}
\it if both $u(0,x)$ and $u(1,x)$ decay fast enough, then $u\equiv 0$.
\end{center} 
These kinds of problems, dealing with conditions ensuring the existence of an identically vanishing solution (or, equivalently,  of a unique solution), are usually known under the name of ``unique  continuation problems''.
Before stating our main results, let us discuss the background of such unique continuation problems and give some physical motivation for   variable coefficient Schr\"odinger equations.

\subsection{Background}
\subsubsection{History of the study on the unique continuation question}

Unique continuation type of questions on partial differential equations (PDEs) has been receiving lots of attention since last century. 
The first unique continuation problem for two dimensional elliptic equations on bounded domains traces back to the work of Carleman \cite{Carleman}, where the so called  ``Carleman estimates'' appeared for the first time. In this direction, people study local unique continuation problems across a smooth oriented hypersurface --- whether solutions vanishing on one side of the surface have to vanish on the other side locally. Such property is in the spirit of Holmgren's uniqueness theorem, in smooth instead of analytic settings. H\"ormander (see \cite{Ho4} and references therein) showed that, given a (smooth) partial differential operator $P$, and a (smooth) hypersurface $\Sigma$ being {\em strongly pseudoconvex} with respect to $P$, then unique continuation properties hold (locally) across $\Sigma$. H\"ormander's theory was later adapted by Isakov and Tataru \cite{IV_JDE, MR1472300} for anisotropic operators including variable-coefficient parabolic and Schr\"odinger operators. In particular, their theories indicate that solutions vanishing on a proper (large enough) space-time open subset of $\bR\times \bR^n$ must be identically zero everywhere.

In the current work, we are interested in a global unique continuation problem. For (variable or constant-coefficient, linear or nonlinear) Schr\"odinger equations, in view of the time-reversibility, one asks:
\begin{question}
Assume two solutions $u_1$ and $u_2$ are close to each other at two different times, say $t=0, 1$. Can we conclude that $u_1 \equiv u_2$?
\end{question}
Such problem was studied for constant-coefficient nonlinear Schr\"odinger equations by Zhang \cite{Z_NLS}, Bourgain \cite{MR1443322}, and Kenig-Ponce-Vega \cite{KPV_CPAM}. Eventually in \cite{IK1, IK2}, Ionescu and Kenig proved that for nonlinear Schr\"odinger equation with constant-coefficient leading terms
\begin{equation*}
    (i\p_t + \Delta) u = G(t, x, u, \bar{u}, \nabla_x u ,\nabla \bar{u})\quad \text{on} \,\, [0,1] \times \bR^n ,
\end{equation*}
and under proper conditions on $G$, if two (sufficiently regular) solutions $u_1$ and $u_2$ are such that $u_1-u_2$ is supported inside some ball $B_R$ at two different times $t=0,1$, then they must be exactly the same.

Later, motivated by Hardy's uncertainty principle \cite{Hardy}, Escauriaza-Kenig-Ponce-Vega replaced the compact support assumption by fast (spatial) Gaussian decay. In a series of papers \cite{EKPV_CPDE, EKPV_MRL, EKPV_JEMS, KPV_CPAM, EKPV_Duke, CEKPV, KPV_MRL, EKPV_JLMS}, they discussed linear constant-coefficient equations with general potentials
\begin{align}\label{eq LS}
 i\partial_t u +  \Delta u =  V(t,x) u, \quad (t,x) \in [0,1] \times \bR^n.
\end{align}
For completeness, we present the statement of their  unique continuation result for classical Schr\"odinger equations with the sharpest decay assumption of this type. This result will be recalled and compared to our result in later sections. 
\begin{customthm}{A}[Theorem 1 in \cite{EKPV_Duke}]\label{thm Sharp}
Assume that $u \in C([0,1], L^2 (\bR^n))$ solves \eqref{eq LS}. Then, under proper boundedness and decay assumptions on $V$, if for $A,B >0$ with $AB > 1 /16$, we have
\begin{equation} \label{eqn-221103-0558}
    \norm{ e^{A \abs{x}^2} u(0,x)}_{L^2 (\bR^n)} + \norm{e^{B \abs{x}^2} u(1,x)}_{L^2 (\bR^n)} < \infty ,
\end{equation}
then $u \equiv 0$. 
\end{customthm}
As a corollary, by considering the  difference equation of the two nonlinear solutions, one obtains similar results for the corresponding nonlinear equations. To motivate such result and see the sharpness of the decay assumption \eqref{eqn-221103-0558}, one considers the free Schr\"odinger flow (that is \eqref{eq LS} with $V=0$)
\begin{align}\label{eq FS}
e^{it\Delta} u_0 = (4 \pi it)^{-n/2} \int_{\bR^n} e^{\frac{i\abs{x-y}^2}{4t}} u_0 (y) \, dy = (2 \pi it)^{-n/2} e^{\frac{i \abs{x}^2}{4t}} \widehat{e^{\frac{i\abs{\cdot}^2}{4t}} u_0} \left(\frac{x}{2t}\right) .
\end{align}
On the one hand, if a solution starts from a Gaussian initial data $u|_{t=0}$, then so is $u|_{t=1}$. On the other hand, solutions to free Schr\"odinger equations cannot have both $u|_{t=0}$ and $u|_{t=1}$ simultaneously decaying at fast Gaussian rates, by Hardy's uncertainty principle.
\begin{customthm}{B}[Hardy's uncertainty principle in \cite{Hardy}]\label{thm Hardy}
 For any function $f : \bR \to \bC$, if the function $f$ itself and its Fourier transform $\widehat{f}$ satisfy 
 \begin{align}
 f (x) = \mathcal{O} (e^{- A x^2}) \quad   \text{and} \quad \widehat{f} (\xi) = \mathcal{O} (e^{- 4B \xi^2})
 \end{align}
with $A,B >0$ and $A B >1/16$, then $f \equiv 0$. 
Moreover, if $A =B =1/4$, then $f (x) =Ce^{- x^2/4}$.
\end{customthm}
In turn, our previous remark about free solutions with Gaussian decay, amounts to a Hardy uncertainty principle for Schr\"odinger equations.
In view of this, Theorem \ref{thm Sharp} can be regarded as a perturbation of Hardy's uncertainty principle for free Schr\"odinger equations, with a zeroth order perturbation term $Vu$.

Similar results for \eqref{eq LS} with gradient terms was proved by Dong-Staubach in \cite{MR2299492}. Unique continuation results were also proved for other dispersive models including discrete Schr\"odinger equations \cite{BV, JLMP}, KdV equations \cite{EKPV_JFA}, and higher order Schr\"odinger equations \cite{HHZ, LY}. 
Finally we would like to mention that there is a research line on Hardy's uncertainty principles for Schr\"odinger equations with magnetic potential. In this direction we refer the interested readers to
\cite{BFGR, CF, BCF} and references therein.
In the current paper, we are interested in perturbations in the leading terms, i.e., linear Schr\"odinger equations with variable coefficients $a_{jk}$.

\subsubsection{Motivation for variable coefficient extensions of Schr\"odinger equations}

The nonlinear Schr\"odinger equation (NLS) is an important model arisen in  physics and it plays a fundamental role in nonlinear optics and many other areas of physics. As the NLS equation allows for soliton solutions, optical solitons are regarded as an important alternative to the next generation of ultrafast optical telecommunication systems. In fact, the problem of soliton control in the nonlinear systems is governed by the NLS equation with variable coefficients, where the variable coefficients represent the group velocity dispersion \cite{HLLXZ}. Further physical motivations of variable coefficient versions of dispersive equations can be found, for example, in \cite{OG, Bis, Ge, KT} and the references therein. In particular, for Schr\"odinger equations with time variable coefficients connected to soliton control, we refer to \cite{FS1,FR, FS2} and references therein.

We would like to emphasize that in the last decades important progress have been made in the study of
Schr\"odinger equations with variable coefficients. Local well-posedness results for semilinear and nonlinear initial value problems have been established thanks to the derivation of the fundamental Strichartz and smoothing estimates. 
For Schr\"odinger equations with the Laplacian replaced by its compactly supported perturbation, Strichartz estimates were proved to be valid in \cite{ST}. This result was later generalized in the case of asymptotically flat perturbations of the Laplacian in \cite{MR2193021} (see also \cite{MR3198586} and references therein). Both the aforementioned results were proved under the so called nontrapping condition (for details about this condition see, for instance, \cite{MR1373768}). Other results about Strichartz estimates in the asymptotically flat case can be found in \cite{MR2565717}, where some nontrapping situations are studied as well.
Homogeneous smoothing estimates for Schr\"odinger equations with space-variable coefficients were proved in \cite{MR1373768} and \cite{MR1361016}. Inhomogeneous smoothing estimates, that is with a gain of one derivative, under the nontrapping condition, were first proved in \cite{MR2166312}. 
Recently, the previous estimates have been investigated also for some time-degenerate variable coefficient Schr\"odinger operators. In this setting, weighted smoothing estimates (both homogeneous and inhomogeneous) and weighted Strichartz estimates were established in \cite{FS1} and \cite{FR} respectively.

Motivated by the variable coefficient extensions of Schr\"odinger equations and the series of works \cite{EKPV_CPDE, EKPV_MRL, EKPV_JEMS, KPV_CPAM, EKPV_Duke, CEKPV, KPV_MRL, EKPV_JLMS}, we would like to investigate the  analogue of unique continuation results for variable coefficient Schr\"odinger equations.

\subsection{Main results}
We now present our main results.
\begin{theorem} \label{thm-220816-1210} 
Let $u\in C^0([0,1], L^2(\bR^n))$ be a solution to \eqref{eqn-220312-0505} with
\begin{equation*}
    A\in C^3(\mathbb{R}^n), \, V=V(x) \in L^\infty(\mathbb{R}^n,\mathbb{R}).
\end{equation*}
Then there exists a small positive number $\epsi_0 = \epsi_0(n,\lambda, \Lambda)$, such that if
\begin{equation} \label{eqn-221104-0447}
     \sup_{\bR^n}  |x||\nabla A| \leq \epsi_0
\end{equation}
and
\begin{equation*}
    e^{k|x|^3}|u(0,x)|,\,\, e^{k|x|^3}|u(1,x)|\in L^2_x(\bR^n),\quad \forall k>0,
\end{equation*}
then $u\equiv 0$.
\end{theorem}

In the second result, we assume the following transversally anisotropic type condition on the leading term coefficients matrix $A(x)$: under a suitable choice of coordinates, for $x = (x_1, x') \in \bR\times \bR^{n-1}$,
	\begin{equation}
	\label{eq Structure}
		A=A(x_1, x')=
		\begin{pmatrix}
			a_{11}(x_1) & 0\\0&\widetilde{A}(x')
		\end{pmatrix},
		\quad \text{where} \,\, \widetilde{A} \,\,\text{is an $(n-1)\times(n-1)$ symmetric matrix}.
	\end{equation}

\begin{theorem} \label{thm-221005-0218}
    Let $u \in C^0([0,1] , L^2 (\bR^n) )$ be a solution to \eqref{eqn-220312-0505} with the above transversally anisotropic type  condition \eqref{eq Structure},
    \begin{equation*}
	a_{11} \in C^2(\bR), 
	\quad \widetilde{A} \in C^3 (\bR^{n-1}),
	\quad V=V(x)\in L^\infty(\mathbb{R}^n,\mathbb{R}).
\end{equation*}
Then there exists a small positive number $\epsi_0 = \epsi_0(n,\lambda, \Lambda)$ and a large number $k=k(n,\lambda, \Lambda, \|a_{11}\|_{C^2}, \| \tilde{A} \|_{C^3})$, such that if
\begin{equation*}
     \sup_{\bR^{n-1}} |x'||\nabla_{x'} \widetilde{A}| \leq \epsi_0,
\end{equation*}
and
	\begin{equation}\label{eqn-220521-0623}
		e^{k|x|^2}|u(0,x)|,\,\, e^{k|x|^2}|u(1,x)|\in L^2_x(\bR^n),
	\end{equation}
then $u\equiv 0$.
\end{theorem}

\begin{remark}
\begin{enumerate}[(1)]
\item
With merely the decay, size, and smoothness conditions on $A$, currently in Theorem \ref{thm-220816-1210}, we obtain the uniqueness under cubic exponential decays. The sharp Gaussian decay rates as in the constant coefficient case are recovered in Theorem \ref{thm-221005-0218}, by assuming \eqref{eq Structure}. Similar conditions also appear in the study of inverse problems under the name of conformally transversally anisotropic condition, cf. \cite[Definition~1.5]{DKSU} and \cite{DKLS}.

\item
In the current paper, we assume the potential $V$ to be real-valued, depending only on $x$. This is needed when we take ``parabolic flow approximations'' in our proof. Actually, we expect that a certain perturbation $V=V_1(x) + V_2(t,x)$ with $V_1$ being real-valued and $V_2$ having fast decay is allowed.
\item As is commonly known, due to the nontrivial behavior of the Hamiltonian flow, the setting of Schr\"odinger equations with variable coefficients is much more complicated to study than the classical constant coefficients one, hence less favorable. Our main assumption \eqref{eqn-221104-0447} is comparable to those in \cite{MR2565717, MR3198586,MR2193021, Tataru, ST} in the sense that the global smallness prevents the possible trapping of the Hamiltonian flow. Actually in terms of smallness at the spatial infinity, our assumption is slightly weaker than asymptotic flatness. It is interesting to ask whether similar unique continuation results still hold if we merely assume the (weak) asymptotic flatness at the spatial infinity, i.e.,
\begin{equation} \label{eqn-221105-1133}
    \limsup_{|x|\rightarrow \infty} |x||\nabla A| < \infty \quad \text{or} \quad \limsup_{|x|\rightarrow \infty} |x||\nabla A| <\epsi \ll 1,
\end{equation}
instead of \eqref{eqn-221104-0447}. We would like to emphasize that under (suitable) weak asymptotic flatness assumptions, Strichartz and local smoothing estimates have been established; see \cite{MR2565717, MR3198586} and references therein for an in-depth analysis of the problem.
Closely related to our uniqueness problem here -- see also \cite{MR3545929} for the backward uniqueness of variable coefficient parabolic equations -- is the so-called Landis-Oleinik conjecture, where \eqref{eqn-221105-1133} instead of our condition \eqref{eqn-221104-0447} is assumed.

\end{enumerate}
\end{remark}

\subsection{Outline of the proofs and organization of the paper}
The major parts of this paper will be devoted to the proof of Theorem \ref{thm-220816-1210}, which we split into the following main steps.

\begin{enumerate}[(1)]
\item 
{\it Persistence of decays at $t=0,1$.}  

In Section \ref{sec-220517-1106}, we study the evolution of the weighted norm $\|e^{\beta|x|^2} u(t,x)\|_{L^2_x}$. In particular, we show that the logarithm of such quantity is almost convex in time. While proving such $log$-convexity, fast decays at any intermediate times ($0 < t < 1$) is needed in order to employ integration by parts, which we do not know a priori. To this end, we utilize the idea in \cite{EKPV_JEMS} and introduce an ``artificial dissipation". Thanks to the decay and regularity inherited from the parabolic estimate, after a proper limiting argument, we are able to close the $log$-convexity computation. Furthermore, by a subordination type inequality, we can prove that any decay faster than Gaussian at initial and terminal times will be preserved for the intermediate time.

\item
{\it Carleman estimates and lower bounds.} 

In Section \ref{sec Carlman}, we derive an absolute lower bound away from $t=0,1$ for non-trivial solutions. Such lower bound is absolute in the sense that the decay power is independent of the decay at initial/terminal times. The proof is by deriving a Carleman inequality with carefully designed weight function, from which the lower bound is obtained by a usual localization procedure.

As mentioned earlier,  inequalities  of Carleman type are widely used in unique continuation principles, where the criticality (and actually, difficulty) falls on hunting suitable weight functions to cooperate assumptions  to be made  on $A (x)$. 
In fact, in our Theorems \ref{thm-220816-1210} and \ref{thm-221005-0218}, two different weights are employed under two different smallness assumptions on  $A$.

\item
{\it A contradiction argument.}

With the upper and lower bounds obtained in previous steps, in Section \ref{sec Main}, we prove the uniqueness in Theorem \ref{thm-220816-1210}. This is by contradiction: the persistence of decay shows that faster decays at $t=0,1$ lead to faster intermediate time decays, while there is an absolute lower bound away from $t=0$ and $t=1$.

\end{enumerate}
In Section \ref{sec-221005-0114}, we prove Theorem \ref{thm-221005-0218} transversally anisotropic type  condition. The proof is very similar to the general case. We first prove that the $\log$-convexity still holds under the new assumption. Then we prove a Carleman estimate with a better ``translated in $x_1$'' weight. Such weight leads to the sharp Gaussian lower bound, and hence the sharp uniqueness theorem.

Besides these, In Section \ref{sec Pre}, we introduce notations and calculate a general formula for commutators that will be used repeatedly (but with possibly different weights) in the current paper. 
In Appendices \ref{sec AppA}, \ref{sec AppB} and \ref{sec AppC}, some technical lemmas including a subordination-type inequality, a Sobolev inequality, and a parabolic approximation result are proved.

\subsection*{Acknowledgement}
Z. L. was partially supported by an AMS-Simons travel grant. Part of this work was done while the third author was in residence at the Institute for Computational and Experimental Research in Mathematics in Providence, RI, during the Hamiltonian Methods in Dispersive and Wave Evolution Equations program. X.Y. was partially supported by an AMS-Simons travel grant. 

\section{Preliminaries}\label{sec Pre}

In this section, we preform a general computation for the symmetric and antisymmetric operators associated with $e^{\phi} (i\p_t + \cL) e^{-\phi}$ (with a very general weight function $\phi = \phi(t,x)$) and their commutator expansion. This will be repeatedly used in later sections.
\subsection{Notations}\label{sec.notations}
In the remainder of the present article, we write
\begin{align}
\inner{f, g}_{L_x^2} :  = \int_{\bR^n} f (x) \overline{g}(x) \, dx \quad \text{and} \quad \inner{f, g}_{L_{t,x}^2} :  = \int_{\bR} \int_{\bR^n} f (t,x) \overline{g}(t,x) \, dx dt .
\end{align}
When the context is clear, we will simple write $\inner{\cdot, \cdot}$ to keep our notations compact. 

We denote the ball centered at the origin with radius $R$ by $B_R :=\{x \in\mathbb{R}^n:\abs{x} \leq R\}$.

For every multi-index $\alpha\in (\mathbb{N}\cup \{0\})^n$, and for every variable coefficient $(n\times n)$ matrix $A=A(x)=\{a_{jk}(x)\}_{j,k=1,\ldots,n}$, we use the notation 
\begin{align}
|\nabla ^\alpha A|:=\left(\sum_{j,k=1}^n |\partial^\alpha a_{jk}|^2\right)^{1/2} ,
\end{align}
where $\partial^\alpha:=\partial_{x_1}^{\alpha_1}\ldots\partial_{x_n}^{\alpha_n}$. 

For every $k\in(\mathbb{N}\cup \{0\})$, we say that $A\in C^k(\mathbb{R}^n)$ if $a_{jk}\in C^k(\mathbb{R}^n)$ for all $j,k=1,\ldots,n$, and we define
\begin{align}
\|A\|_{C^k}:=\sum_{i=1}^k \sup_{\substack{x\in\mathbb{R}^n,\\ 
\alpha\in (\mathbb{N}\cup \{0\})^n, \, |\alpha|=i}} |\nabla^\alpha A|,\quad \text{where}\quad  |\alpha|=\alpha_1+\cdots+\alpha_n.
\end{align}

Finally, we shall often use the abbreviations
\begin{align}
 \int \,f dx:=\int_{\mathbb{R}^n} f(x)\, dx, \quad  \iint g \, dx dt:=\int_{\mathbb{R}}\int_{\mathbb{R}^n}g(t,x)\,dx dt.
\end{align}

\subsection{General formulas for commutator computation}\label{sub-220817}
Let $\phi=\phi(t,x)$ be a general weight function and $\cL$ be given in \eqref{eqn-220517-0913}. We decompose
\begin{equation} \label{eqn-220815-0801}
    e^{\phi} (i\p_t + \cL) e^{-\phi} = \cS + \cA,
\end{equation}
where
\begin{equation} \label{eqn-220815-0801-1}
\begin{split}
    \cS & := i\p_t + a_{kj}\p_{kj}^2 + (\p_k a_{kj})\p_j + (\p_k \phi)(\p_j \phi)a_{kj},\\
    \cA & := -2 (\p_l\phi)a_{ml}\p_m - (\p_l\phi) (\p_m a_{ml}) - (\p_{ml}^2\phi) a_{ml} - i\p_t\phi .
\end{split}
\end{equation}
are the symmetric and antisymmetric parts with respect to the inner product $\left< \cdot , \cdot \right>_{L^2_{t,x}} $.

A direct computation shows that formally the commutator $ [\cS, \cA]$ is given by
\begin{align}
    [\cS, \cA] =\cS\cA - \cA\cS = T_2 + T_1 + T_{0,1} + T_{0,2},
\end{align}
where
\begin{align}
T_2 & := -4 a_{kj}a_{ml}(\p_{kl}^2\phi) \p_{ml}^2 - 4 a_{kj} (\p_k a_{ml}) (\p_l \phi) \p_{mj}^2 + 2 a_{ml} (\p_m a_{kj}) (\p_l \phi) \p_{kj}^2,\\
    T_1 
    & :=
    -4i a_{ml} (\p_t\p_l\phi)  \p_m - 4 a_{kj} a_{ml} (\p^3_{kjl} \phi) \p_m
    \\&\quad 
    -4a_{kj} (\p_{kl}^2\phi) (\p_j a_{ml}) \p_m -4 a_{ml} (\p^2_{jl}\phi) (\p_k a_{kj}) \p_m - 2 a_{kj} (\p^2_{ml}\phi) (\p_k a_{ml}) \p_j 
    \\&\quad
    - 2 (\p_l\phi) (\p_j a_{ml}) (\p_k a_{kj}) \p_m - 2 a_{kj} (\p_l\phi) (\p^2_{km} a_{ml}) \p_j + 2 a_{ml} (\p_l \phi) (\p^2_{km} a_{kj}) \p_j,
    \\
    T_{0,1} &:= 4 a_{ml} a_{kj} (\p_l\phi) (\p_{km}^2 \phi) (\p_j \phi) + 2 a_{ml}(\p_m a_{kj}) (\p_l \phi) (\p_k\phi) (\p_j \phi),\\
    T_{0,2}
    &:=
    - a_{kj} a_{ml} (\p_{kjml}^4 \phi) - 2 i a_{kj} (\p_{kj} \p_t \phi) + \p_{tt}\phi
    \\&\quad
    - 2i (\p_t \p_l \phi) (\p_m a_{ml}) - 2 a_{kj}(\p^3_{kjl} \phi) (\p_m a_{ml}) - 2 a_{kj} (\p^3_{kml}\phi) (\p_j a_{ml})
    \\&\quad
    - (\p_k a_{kj}) (\p^2_{jl} \phi) (\p_m a_{ml}) - (\p_k a_{kj}) (\p^2_{ml}\phi) (\p_j a_{ml}) - 2 a_{kj} (\p^2_{kl} \phi) (\p^2_{jm} a_{ml}) - a_{kj}(\p^2_{ml}\phi)(\p^2_{kj}a_{ml})
    \\&\quad
    - (\p_k a_{kj})(\p_l \phi) (\p^2_{jm} a_{ml}) - a_{kj}(\p_l\phi) (\p^3_{kjm}a_{ml}).\label{eqn-1103-1042}
\end{align}
Note that here $T_2$, $T_1$ are of order 2 and 1 respectively, while $T_{0,1}$ and $T_{0,2}$ are order 0 terms. 
The (complicated) explicit expressions for $T_1, T_{0,2}$ are only needed in Section \ref{sec-221005-0114}, where some cancellation needs to be verified. Everywhere else, it suffices to work with the following:
\begin{align}\label{eqn-220815-0801-2}  
\begin{aligned}
       T_1 & : = -  4 i (\p_t \p_l \phi) a_{ml}\p_m - 4 a_{kj} a_{ml} (\p_{kjl}^3 \phi) \p_m + O(1) (|\nabla^2\phi| |\nabla A| + |\nabla \phi| |\nabla A|^2 + |\nabla \phi| |\nabla^2 A| )\nabla,\\
    T_{0,2} & :=  - a_{kj} a_{ml} (\p_{kjml}^4 \phi) - 2i a_{kj} (\p_{kj} \p_t \phi) + \p_{tt}\phi \\
    & \quad + O(1) (|\nabla \p_t \phi| |\nabla A| + |\nabla^3 \phi| |\nabla A| + |\nabla^2 \phi| |\nabla A|^2 + |\nabla^2 \phi| |\nabla^2 A| + |\nabla \phi| |\nabla A| |\nabla^2 A| + |\nabla \phi| |\nabla^3 A| ), 
\end{aligned}
\end{align}
where
\begin{equation*}
    \nabla = \nabla_x
\end{equation*}
and $O(1)g$ stands for a function which is bounded from above by $C(\Lambda,n)|g|$.

\section{$Log$-convexity with quadratic weight and persistence of decay}\label{sec-220517-1106}
The goal of this section is to prove that if $u$ solves equation \eqref{eqn-220312-0505}, then the function $H:[0,1]\rightarrow \mathbb{R}$, defined as
\begin{align}
H(t):=\|e^{\beta|x|^2}u(t, x)\|^2_{L^2_{x}},
\end{align}
is $log$-convex. More specifically, we will focus on establishing the following result.
\begin{theorem}\label{thm-220515-0921}
	Let $u\in C([0,1],L^2 (\mathbb{R}^n))$ be a solution to \eqref{eqn-220312-0505} with 
	\begin{align}
	    A\in C^3(\mathbb{R}^n), \quad V=V(x)\,\,\text{being real-valued},\quad  \text{and}\quad M_1 := \|V\|_{L^\infty} < \infty.
	\end{align}
	Then there exist a small enough $\epsi_0 = \epsi_0(n,\lambda, \Lambda) >0$ and a large enough $\beta_0 = \beta_0(n,\lambda, \Lambda, \|A\|_{C^3})$, such that if
\begin{equation}\label{eqn-220512-1132}
    \sup_{\bR^n}|x||\nabla A| \leq \epsi_0
\end{equation}
	and
	\begin{equation}\label{eqn-220704-0504}
		e^{\beta|x|^2}u(0,x), \, e^{\beta|x|^2}u(1,x) \in L^2(\bR^n), \quad \text{with} \,\, \beta>\beta_0,
	\end{equation}
then for all $t\in(0,1)$, we have 
\begin{equation}\label{eqn-220711-0221}
		\|e^{\beta|x|^2}u(t,x)\|_{L^2_x}^2 
		\leq
		C e^{ M_1^2 } (\|e^{\beta|x|^2}u(0,x)\|_{L^2_x}^2)^{1-t} (\|e^{\beta|x|^2}u(1,x)\|_{L^2_x}^2)^t,
	\end{equation}
  and
	\begin{equation}\label{eqn-220711-0222}
	\beta\|\sqrt{t(1-t)}e^{\beta|x|^2}|\nabla u|\|_{L^2_{t,x}}^2 + \beta^3\|\sqrt{t(1-t)}e^{\beta|x|^2}|xu|\|_{L^2_{t,x}}^2 \leq 
	 Ce^{M_1^2} ( \|e^{\beta|x|^2}u(0,x)\|^2_{L^2_{x}} + \|e^{\beta|x|^2}u(1, x)\|^2_{L^2_{x}} ).
	\end{equation}
	Here $C$ is an absolute constant.
	
	In particular, when $A=I_n$, we can take $\beta_0=0$.
\end{theorem}

From Theorem \ref{thm-220515-0921}, we can further derive the persistence of higher order decays.
\begin{corollary}\label{cor-220515-1117}
	Under the assumptions of Theorem \ref{thm-220515-0921}, for any $\alpha>1$, there exists some $\kappa_0 = \kappa_0 (\beta_0,\alpha)$, such that if we further assume that for some $\kappa > \kappa_0$
	\begin{equation*}
	e^{\kappa |x|^{2\alpha}}u(0,x), \,  e^{\kappa |x|^{2\alpha}} u(1,x) \in L^2(\bR^n),
	\end{equation*}
	then for any $t\in (0,1)$, we also have 
	\begin{equation*}
	e^{\kappa |x|^{2\alpha}}u(t,x)\in L^2(\bR^n)
	\end{equation*}
	and
	\begin{equation*}
	\sqrt{t(1-t)}e^{\kappa |x|^{2\alpha}}  \nabla u(t,x), \,  \sqrt{t(1-t)}e^{\kappa |x|^{2\alpha}} xu(t,x) \in L^2([0,1],L^2(\mathbb{R}^n)),
	\end{equation*}
with estimates
 	\begin{equation}\label{cor.high.order.decay.1}
	    \int_{\bR^n} e^{\kappa |x|^{2\alpha}} |u(t,x)|^2 \,dx \leq C \left( \int_{\bR^n} e^{\kappa |x|^{2\alpha}} |u_0|^2 \, dx \right)^{1-t} \left( \int_{\bR^n} e^{\kappa |x|^{2\alpha}} |u_1|^2 \, dx \right)^t
	\end{equation}
	and 
	\begin{equation} \label{cor.high.order.decay.2}
	\|\sqrt{t(1-t)}e^{\kappa |x|^{2\alpha}} \nabla u \|_{L^2_{t,x}}^2 + \|\sqrt{t(1-t)}e^{\kappa |x|^{2\alpha}} xu \|_{L^2_{t,x}}^2 
	\leq C
	( \|e^{ \kappa |x|^{2\alpha} } u(0,x)\|^2_{L^2_{x}} + \|e^{ \kappa |x|^{2\alpha} }u(1,x)\|^2_{L^2_{x}} ).
	\end{equation}
\end{corollary}
\begin{remark}\label{rmk-220711}
The proof of Theorem \ref{thm-220515-0921} is based on the lower bound for the operator
\begin{equation*}
	[\cS,\cA]
\end{equation*}
in  \eqref{eqn.comm.convexity}. By assuming a certain regularity and a Gaussian decay at any intermediate time of the solution, such convexity result is proved in Lemma \ref{lem-220711}. However, such decay and regularity are not known {\it a priori}. To this end, we introduce an ``artificial dissipation'' by considering the parabolic flow approximation $u_\epsi = e^{t \varepsilon (\cL + V) u}$ (see \eqref{eq u_e}), then use the parabolic estimate to get the desired decay and regularity of $u_\varepsilon$, and hence the log-convexity of the weighted norm. Finally, passing $\varepsilon\rightarrow 0$, we get the result for $u$ solving the original equation \eqref{eqn-220312-0505}. We remark that a similar strategy was also used in \cite{EKPV_JEMS} to prove Theorem \ref{thm-220515-0921} in the case when $a_{jk}=\delta_{jk}$.
\end{remark}
For the reasons explained in Remark \ref{rmk-220711},
we postpone the proof of Theorem \ref{thm-220515-0921} till the end of the section (Section \ref{ssec 3.2}) and devote the subsection below (Section \ref{sec-220915-0640}) to the proof of the results concerning the ``dissipative case''.

\subsection{$Log$-convexity for equations with dissipation} \label{sec-220915-0640}
In this subsection we study the equation
\begin{equation}\label{eqn-220630-1118}
\p_t u = (a+ib)(\cL u + Vu),\quad (t,x)\in [0,1]\times \bR^n.
\end{equation}

We prove the $log$-convexity property in Proposition \ref{prop-220704-0508}, with dissipation.
\begin{proposition}\label{prop-220704-0508}
	Let $u \in C^0([0,1],L^2(\mathbb{R}^n)) \cap L^2([0,1],H^1(\mathbb{R}^n))$ be a solution to
	\begin{equation}\label{eqn-220705-0804}
		\p_t u = (a + ib)(\cL u + Vu + g(t,x)), \quad (t,x) \in [0,1] \times \bR^n
	\end{equation}
	with $A \in C^3(\bR^n)$,
	\begin{equation} \label{eqn-220921-1258}
		a > 0, \quad b \in \bR, \quad \text{and} \,\, V = V (x) \,\,\text{being real-valued}.
	\end{equation}
	There exist a small enough $\epsi_0 = \epsi_0(n,\lambda, \Lambda) >0$   and a large enough $\beta_0 = \beta_0(n,\lambda,\Lambda, \|A\|_{C^3})$, such that if \eqref{eqn-220512-1132}-\eqref{eqn-220704-0504} are satisfied, then the following estimates hold  $e^{\beta|x|^2}g\in L^\infty([0,1],L^2(\mathbb{R}^n))$
	\begin{equation}\label{eqn-220705-1232}
		\|e^{\beta|x|^2}u(t,x)\|_{L^2_x}^2 \leq 	C e^{(a^2 + b^2) (1 + M_1^2 + M_2^2)} (\|e^{\beta|x|^2}u(0,x)\|_{L^2_x}^2)^{1-t} (\|e^{\beta|x|^2}u(1,x)\|_{L^2_x}^2)^t,
	\end{equation}
	\begin{equation}\label{eqn-220521-0546}
	\begin{split}
	&\beta\|\sqrt{t(1-t)}e^{\beta|x|^2}|\nabla u|\|_{L^2_{t,x}}^2 + \beta^3\|\sqrt{t(1-t)}e^{\beta|x|^2}|xu|\|_{L^2_{t,x}}^2 
	\\&\leq 
	C\big( ( a^2 + b^2)^{-1} + (1 + M_1^2 + M_2^2)  e^{(a^2 + b^2) (1 + M_1^2 + M_2^2)} \big)  \max\{H(0),H(1)\}.
	\end{split}
	\end{equation}
	Here, 
	\begin{equation}\label{eqn-220705-0739}
	M_1 := \|V\|_{L^\infty_{t,x}} \quad  \text{and} \quad M_2 := \sup_{t} \frac{ \|e^{\beta|x|^2} g \|_{L^2_x} }{ \| e^{\beta|x|^2} u \|_{L^2_x} }.
	\end{equation}
\end{proposition}

\begin{remark}\label{rmk-1103-1427}[Outline of the proof of Proposition \ref{prop-220704-0508}]
We split the proof of this result in three steps, treated, respectively, in three Subsections:
\begin{itemize}
    \item[\textbf{Step 1}] We prove a $log$-convexity result under some decay assumptions on the solution; this is Lemma \ref{lem-220711} in Subsection \ref{subsub-220711-1}.
    \item[\textbf{Step 2}] We prove decay properties for solutions to dissipative equations at any intermediate time $t\in (0,1)$ in Subsection \ref{subsub-220711-2}; 
    \item[\textbf{Step 3}] We use Step 2 to derive the decay rate required in Step 1 and conclude the proof of Proposition \ref{prop-220704-0508} in Subsection \ref{subsub-220711-3}.
\end{itemize}
\end{remark}

\subsubsection{$Log$-convexity for equations with dissipation under decay assumptions}\label{subsub-220711-1}
In order to prove Proposition \ref{prop-220704-0508} we need several lemmas, of which the main one is the following {\it a priori} estimate.

\begin{lemma}\label{lem-220711}
The conclusions of Proposition \ref{prop-220704-0508} hold if we further assume that $u  \in L^2_{t,loc}((0,1), H^2_{x,loc}(\mathbb{R}^n))$ and for each $t \in (0,1)$ there exists some small constant $\beta_\epsi = \beta_\epsi(t)>0$,
	\begin{equation}\label{eqn-220705-0802}
	|u(t,x)| + |\nabla_x u(t,x)| = o(e^{-(\beta + \beta_\epsi)|x|^2}) \quad \text{as} \,\, |x| \rightarrow \infty.
  \end{equation}
  In this case, assumption \eqref{eqn-220921-1258} can be relaxed to
  \begin{equation*}
      a, b \in \bR, \, a^2 + b^2 \neq 0, \quad \text{and} \,\, V = V (t,x) \in L^\infty([0,1]\times \mathbb{R}^n) \,\,\text{not necessarily being real-valued}.
  \end{equation*}
\end{lemma}

It is worth mentioning that with the extra decay assumption \eqref{eqn-220705-0802}, $a>0$ is not needed here.
The proof of Lemma \ref{lem-220711} is essentially obtained by combining the commutator lower bound \eqref{eqn-220515-0938} and the following consequence of Poincaré inequality, which we prove in Appendix \ref{sec AppB} for completeness.
\begin{lemma}\label{lem-220515-0957}
	For any $f=f(x)\in H^1(B_{2r})$ and any $r>0$,
	\begin{equation*}
	\|f\|_{L^2(B_{r})} \leq C\left(r\|\nabla f\|_{L^2(B_{2r})} + r^{-1}\|xf\|_{L^2(B_{2r})}\right),
	\end{equation*}
	where $C=C(n)$.
\end{lemma}
Now we are ready to give the proof of Lemma \ref{lem-220711}.
\begin{proof}[Proof of Lemma \ref{lem-220711}]
	
	Below the computations until \eqref{eqn-220704-0828-2} basically follow the abstract framework in \cite[Lemma~2]{EKPV_JEMS}. Here, for completeness we provide a proof with explicit computations. Note that in the following, all the formal integration by parts are justified due to our decay assumption \eqref{eqn-220705-0802}.
	Denote
	\begin{equation}\label{eqn-220611-0528}
	f(t,x):=e^{\beta|x|^2}u(t,x),\quad H(t):=\left<f,f\right>_{L^2_x} , \quad h(t,x) := e^{\beta |x|^2} g(t,x),
	\end{equation}
	where $u$ and $g$ are as in \eqref{eqn-220705-0804}. 
	In this proof, we denote
	\begin{equation*}
	\left<\cdot,\cdot\right>:=\left<\cdot,\cdot\right>_{L^2_x},
	\end{equation*}
	and call $\cS$ and $\cA$ the symmetric and antisymmetric operators in the decomposition
    \begin{equation}\label{eqn-220515-0923}
	e^{\beta |x|^2} \cL(e^{-\beta |x|^2}) = \cS + \cA
    \end{equation}
    with respect to the $L^2_x$-inner product (instead of $L^2_{t,x}$ as in \eqref{eqn-220815-0801}).
	Now, since $u$ is a solution to \eqref{eqn-220705-0804} and our weight $e^{\beta|x|^2}$ is independent of $t$,
	\begin{equation}\label{eqn-220609-1105}
	\begin{split}
	\p_t f 
	&= 
	(a + ib) (e^{\beta|x|^2}\cL e^{-\beta|x|^2} + V) f 
	= 
	(a + ib)(\cS+\cA)f + (a + ib)( V f + h)\\
	&=
	\underbrace{(a\cS + ib\cA) }_{symmetric} f + \underbrace{(ib\cS + a\cA) }_{antisymmetric} f  + (a + ib)( V f + h).
	\end{split}
	\end{equation}
	Therefore, 
	\begin{align}
	\frac{dH}{dt} 
	&= 
	2 \mathrm{Re} \left<\p_t f,f\right>  
	=
	2 \left<(a\cS + ib\cA) f , f\right> +  2 \mathrm{Re} \left< (a + ib)( V f + h), f \right>. \label{eqn-220609-1118}
	\end{align}
	We define
	\begin{equation}\label{eqn-220704-0848}
	D(t) := \left<(a\cS + ib\cA) f , f\right> \quad\text{and}\quad  N(t):=\frac{D(t)}{H(t)}.
	\end{equation}
	Now, since both $\cS$ and $\cA$ have coefficients independent of $t$, we have
	\begin{align*}
	\frac{dD}{dt} 
	&= 
	\left< (a\cS + ib\cA) \partial_tf , f \right> + \left< (a\cS + ib\cA) f , \p_tf \right>.
	\end{align*}
	Using \eqref{eqn-220609-1105} and the symmetry and antisymmetry of the operators $\cS$ and $\cA$, we can further compute
	\begin{align}
	\frac{dD}{dt} 
	&=
	\left< (a\cS + ib\cA) f  + (ib\cS + a\cA)f + (a + ib)( V f + h) , (a\cS + ib\cA) f \right>   \\
	&\quad + \left< (a\cS + ib\cA) f , (a\cS + ib\cA) f  + (ib\cS + a\cA)f + (a + ib)( V f + h) \right>  
	\\&=
	\left< [a\cS + ib\cA, ib\cS + a\cA] f, f \right> + 2 \| (a\cS + ib\cA)f \|_{L_x^2}^2 + 2 \mathrm{Re} \left< (a\cS + ib\cA) f, (a + ib)( V f + h)\right>
	\\&=
	(a^2 + b^2) \left< [\cS, \cA] f, f \right> + 2 \| (a\cS + ib\cA)f \|_{L_x^2}^2 + 2 \mathrm{Re} \left< (a\cS + ib\cA) f, (a + ib)( V f + h)\right>. \label{eqn-220611-0601}
	\end{align}
		Hence, 
	\begin{align}
	\frac{dN}{dt}
	&=
	\frac{D'H - DH'}{H^2}  
	\\&=
	\frac{(a^2 + b^2) \left< [\cS, \cA] f, f \right>}{H}  + 2 \frac{ \| (a\cS + ib\cA)f \|_{L_x^2}^2 + \mathrm{Re} \left< (a\cS + ib\cA) f, (a + ib)( V f + h) \right>}{H} - \frac{DH'}{H^2}. \label{eqn-220515-1029}
	\end{align}
	By the bilinearity and the polarization identity,
	\begin{align}
	\| (a\cS + ib\cA)f \|_{L_x^2}^2 &+ \mathrm{Re} \left< (a\cS + ib\cA) f, (a + ib)( V f + h) \right>
	\\&=
	\mathrm{Re} \left< (a\cS + ib\cA)f , (a\cS + ib\cA)f + (a + ib)( V f + h) \right>
	\\&=
	\frac{1}{4} \left(\| 2(a\cS + ib\cA)f + (a + ib)( V f + h) \|_{L_x^2}^2 - \| (a + ib)( V f + h) \|_{L_x^2}^2 \right). \label{eqn-220704-0828-1}
	\end{align}
	From \eqref{eqn-220609-1118} and the bilinearity,
	\begin{align}
		DH'
		&=
		\left<(a\cS + ib\cA)f , f \right> \left( 2 \left<(a\cS + ib\cA)f , f \right> + 2 \mathrm{Re}\left< (a + ib)( V f + h), f \right> \right)
		\\&=
		\frac{1}{2} \left( \mathrm{Re} \left< 2(a\cS + ib\cA) f + (a + ib) ( V f + h) , f \right> - \mathrm{Re} \left< (a + ib) ( V f + h) , f \right> \right)
		\\ &\qquad
		\times \left( \mathrm{Re} \left< 2(a\cS + ib\cA) f + (a + ib) ( V f + h) , f \right> + \mathrm{Re} \left< (a + ib) ( V f + h) , f \right> \right)
		\\ &=
		\frac{1}{2} \left( (\mathrm{Re} \left< 2(a\cS + ib\cA) f + (a + ib) ( V f + h) , f \right>)^2 - (\mathrm{Re} \left< (a + ib) ( V f + h) , f \right>)^2 \right). \label{eqn-220704-0828-2}
	\end{align}
	Substituting \eqref{eqn-220704-0828-1} - \eqref{eqn-220704-0828-2} back to \eqref{eqn-220515-1029}, then using the Cauchy-Schwarz inequality, the triangle inequality, and \eqref{eqn-220705-0739}, we have
	\begin{align}
	\frac{d N}{d t}
	&=
	\frac{(a^2 + b^2) \left< [\cS, \cA] f, f \right>}{H}
	+ 
	\frac{1}{2 H} \left( \| 2(a\cS + ib\cA)f + (a + ib)( V f + h)\|_{L_x^2}^2 - \| (a + ib)( V f + h) \|_{L_x^2}^2 \right) 
	\\&\quad-
	\frac{1}{2 H^2} \left( (\mathrm{Re} \left< 2(a\cS + ib\cA) f + (a + ib) ( V f + h) , f \right>)^2 - (\mathrm{Re} \left< (a + ib) ( V f + h) , f \right>)^2 \right)
	\\&\geq
	\frac{(a^2 + b^2) \left< [\cS, \cA] f, f \right>}{H}
	+ 
	\frac{1}{2 H} \left( \| 2(a\cS + ib\cA)f + (a + ib)( V f + h)\|_{L_x^2}^2 - \| (a + ib)( V f + h) \|_{L_x^2}^2 \right) 
	\\&\quad-
	\frac{1}{2 H^2} (\mathrm{Re} \left< 2(a\cS + ib\cA) f + (a + ib) ( V f + h) , f \right>)^2
	\\&\geq
	\frac{(a^2 + b^2) \left< [\cS, \cA] f, f \right>}{H}
	+ 
	\frac{1}{2 H} \left( \cancel{ \| 2(a\cS + ib\cA)f + (a + ib)( V f + h)\|_{L_x^2}^2 } - \| (a + ib)( V f + h) \|_{L_x^2}^2 \right) 
	\\&\quad-
	\cancel{ \frac{1}{2 H^{2}} \| 2(a\cS + ib\cA) f + (a + ib)( V f + h) \|_{L_x^2}^2 \|f\|_{L_x^2}^2 }
	\\&\geq 
	\frac{(a^2 + b^2) \left< [\cS, \cA] f, f \right>}{H} - (a^2 + b^2) (M_1^2 + M_2^2). \label{eqn-220704-0846}
\end{align}
	
	Next, we bound the term involving $[\cS,\cA]$ from below. Note that, although here we are computing with respect to $\left<\cdot,\cdot\right>_{L^2_x}$ instead of $\left<\cdot,\cdot\right>_{L^2_{t,x}}$, the results in \eqref{eqn-220815-0801-1} - \eqref{eqn-220815-0801-2} are still applicable due to  the weight being time-independent. Therefore we obtain
	\begin{align}
	[\cS,\cA] 
	&= 
	[a_{kj}\p_{kj}^2 + (\p_ka_{kj})\p_j + 4\beta^2 a_{kj}x_kx_j, -4\beta a_{ml}x_l\p_m - 2\beta(\p_ma_{ml})x_l - 2\beta\tr(A)]\\
	&=
	-8\beta a_{kj}a_{ml}(\underbrace{\p_k x_l}_{\delta_{kl}})\p_{mj}^2 - 8\beta a_{kj}(\p_k a_{ml})x_l\p_{mj}^2 + 4\beta a_{ml}x_l(\p_m a_{kj})\p_{kj}^2\\
	&\quad +
	O(1)\beta\big(|x||\nabla^2 A| + |x||\nabla A|^2 + |\nabla A|\big)\nabla\\
	&\quad +
	32\beta^3 a_{ml}a_{kj}(\underbrace{\p_m x_k}_{\delta_{mk}})x_lx_j + 16\beta^3 a_{ml}x_l\p_m(a_{kj})x_kx_j
	\\&\quad+
	O(1)\beta(|x||\nabla^3 A| + |x||\nabla^2 A||\nabla A| + |\nabla^2 A| + |\nabla A|^2). \label{eqn.comm.convexity}
\end{align}
Now, since $f=e^{\beta |x|^2}$ along with its derivative decay fast enough, integration by parts in the second order terms yields 
\begin{align}
	\left<[\cS,\cA]f, f\right>
	&= 
	8\beta\int a_{kj}a_{mk}\p_m f\p_j f \,dx+ \underbrace{8\beta\int a_{kj}(\p_k a_{ml})x_l\p_m f \p_j f\,dx - 4\beta\int a_{ml}x_l (\p_m a_{kj})\p_k f \p_j f \,dx}_{II_1}\\
	&\quad -
	O(1) \underbrace{\beta\int\big(|x||\nabla^2 A| + |x||\nabla A|^2 + |\nabla A|\big)|\nabla f||f| \,dx}_{I}\\
	&\quad +
	32\beta^3 \int a_{kj}a_{kl}x_lx_j|f|^2 \,dx+ \underbrace{16\beta^3 \int a_{ml}x_l\p_m(a_{kj})x_kx_j|f|^2\,dx}_{II_2} \\
	&\quad - O(1)\beta \int\big(|x||\nabla^3 A| + |x||\nabla^2 A||\nabla A| + |\nabla A|^2 + |\nabla^2 A|\big)|f|^2\,dx. \label{eqn-220508-0350-1}
\end{align}
Noting that $a_{kj}a_{mk} = (A^2)_{mj}$, the ellipticity condition in \eqref{eqn-221030-0836} gives
\begin{equation*}
	\lambda_{min}(A^2)\geq \lambda_{min}(A)^2 \geq \lambda^2, \quad \text{i.e.,}\,\, a_{kj}a_{mk} \geq \lambda^2\delta_{mj},
\end{equation*}
which implies
\begin{equation}\label{eqn-220508-0349-11}
	a_{kj}a_{mk}\p_m f\p_j f \geq \lambda^2 |\nabla f|^2 ,\quad a_{kj}a_{kl}x_lx_j \geq \lambda^2|x|^2.
\end{equation}
Using Young's inequality, we have
\begin{equation}\label{eqn-220508-0349-22}
	I \leq \beta\int \lambda^2 |\nabla f|^2 \, dx  + C\beta \int\big(|x|^2|\nabla^2 A|^2 + |x|^2 |\nabla A|^4 + |\nabla A|^2\big)|f|^2\,dx.
\end{equation}
Moreover, from the boundedness of $a_{jk}$ in \eqref{eqn-221030-0836},
\begin{equation} \label{eqn-221030-0838}
|II_1| \leq C\beta\Lambda \int |x||\nabla A||\nabla f|^2 \, dx  \quad \text{and}\quad |II_2| \leq C\beta\Lambda \int |x||\nabla A||x f|^2 \, dx .
\end{equation}
Substituting \eqref{eqn-220508-0349-11}, \eqref{eqn-220508-0349-22}, and \eqref{eqn-221030-0838} back to \eqref{eqn-220508-0350-1}, we get
\begin{align}\label{eqn-220515-0938}   
\begin{aligned}
 		&\left<[\cS,\cA]f,f\right>\\ 
		&\geq 
		\beta\int \big(8\lambda^2 - C \Lambda  |x||\nabla A| - \lambda^2\big)|\nabla f|^2\,dx\\
		&\quad +
		\beta^3 \int \left(32\lambda^2 - C \Lambda   |x||\nabla A| - \frac{C}{\beta^2}\left(\frac{|\nabla^3 A|}{|x|} + |\nabla^2 A|^2 + |\nabla A|^4 + \frac{|\nabla^2 A||\nabla A|}{|x|} + \frac{|\nabla^2 A|}{|x|^2} + \frac{|\nabla A|^2}{|x|^2}\right)\right)|xf|^2\,dx, 
\end{aligned}
\end{align}
where $C = C(\lambda,\Lambda, n)$. Hence, if we choose $\epsi_0$ in \eqref{eqn-220512-1132} small enough, such that
\begin{equation}
    C \Lambda  |x||\nabla A| \leq \lambda^2,
\end{equation}
then for each $t\in (0,1)$, we have
	\begin{equation}\label{eqn-220515-1012}
	\left<[\cS,\cA]f(t,\cdot),f(t,\cdot)\right> \geq 6\beta\lambda^2\int_{\bR^n}|\nabla f(t,x)|^2\,dx + 31\beta^3\lambda^2\int_{\bR^n} |xf(t,x)|^2\,dx - III,
	\end{equation}
	where
	\begin{align*}
	III 
	&= 
	C \beta\int_{\bR^n} (|\nabla^2 A|^2 + |\nabla A|^4)|x|^2|f|^2\,dx + C \beta\int_{\bR^n} (|\nabla^3 A| + |\nabla^2 A||\nabla A|)|x||f|^2 \,dx \\
	& \quad + \underbrace{C \beta\int_{\bR^n} (|\nabla^2 A| + |\nabla A|^2)|f|^2\,dx}_{IV}.
	\end{align*}
	For some $r>0$ to be fixed later, splitting
	\begin{equation*}
	IV \leq C \beta\int_{|x|\geq r} (|\nabla^2 A| + |\nabla A|^2)|x/r|^2|f|^2 \,dx + C \beta\int_{|x|\leq r} (|\nabla^2 A| + |\nabla A|^2)|f|^2\,dx,
	\end{equation*}
	and using Young's inequality, we obtain
	\begin{align}
	III
	&\leq
	C \big( \beta\sup_{\bR^n}(|\nabla^2 A|^2 + |\nabla A|^4) + \beta^2\sup_{\bR^n}(|\nabla^3 A| + |\nabla^2 A||\nabla A|)^2 + \beta r^{-2}\sup_{B_r^{\mathsf{c}}}(|\nabla^2 A| + |\nabla A|^2)\big)\int_{\bR^n}|xf|^2\,dx   
	\\&\quad+ \int_{\bR^n} |f|^2\,dx + C \beta\sup_{B_r}(|\nabla^2 A| + |\nabla A|^2)\int_{|x|\leq r} |f|^2\,dx \label{eqn-220515-1011}
	\\&\leq
	C_1 (\beta + \beta^2 + \beta r^{-2}) \int_{\bR^n}|xf|^2\,dx + H(t) + C_1 \beta \big( \int_{B_{2r}}(r^2 |\nabla f|^2 + r^{-2} |xf|^2)\,dx \big), \label{eqn-220906-0255}
	\end{align}
	where $C_1$ is a constant depending on $(\lambda, \Lambda, n, \|A\|_{C^3})$.
	Note that to get the last inequality we applied Lemma \ref{lem-220515-0957} to $f(t,\cdot)$.
Substituting \eqref{eqn-220906-0255} back to \eqref{eqn-220515-1012}, if we choose
\begin{equation}
    r = \lambda / \sqrt{C_1} \quad \text{and} \quad \beta_0 = \beta_0 ( \lambda, \Lambda, n, \|A\|_{C^3} )\,\,\text{large enough}
\end{equation}
and such that
\begin{equation*}
    C_1 r^2 < \lambda^2, \quad C_1 (\beta_0 + \beta_0^2 + 2\beta_0 r^{-2}) < \beta_0^3\lambda^2,
\end{equation*}
then for any $\beta > \beta_0$,
	\begin{align} \label{eqn-220515-1028-1}
	\left<[\cS,\cA]f(t,\cdot), f(t,\cdot)\right> 
	\geq 
	5\beta\lambda^2\int_{\bR^n}|\nabla f(t,x)|^2\,dx + 30\beta^3\lambda^2\int_{\bR^n}|xf(t,x)|^2\,dx - H(t).
	\end{align}
	
	Now, using \eqref{eqn-220609-1118} and \eqref{eqn-220704-0848}, then substituting \eqref{eqn-220515-1028-1} back to \eqref{eqn-220704-0846}, we have
	\begin{align}
	\frac{d}{dt}&{\left(\frac{d\log(H(t))}{dt} - 2\frac{2 \mathrm{Re} \left< (a + ib)( V f + h), f \right>}{H}\right)} 
	\\=&
	2\frac{dN}{dt} 
	\\ \geq&
	\frac{2 (a^2 + b^2)}{H} \left( 5\beta\lambda^2\int_{\bR^n}|\nabla f(t,x)|^2\,dx + 30 \beta^3\lambda^2\int_{\bR^n}|xf(t,x)|^2\,dx - H(t) \right) - 2 (a^2 + b^2) (M_1^2 + M_2^2)
	\\ \geq &
	- 2 (a^2 + b^2) (1 + M_1^2 + M_2^2).
	\end{align}
	Hence, defining
	\begin{equation*}
	\widetilde{M}(t):=\int_0^t \frac{2 \mathrm{Re} \left< (a + ib)( V f + h)(s,x), f(s,x)\right>_{L^2_x} }{H(s)} \, ds,
	\end{equation*}
	we obtain
	\begin{align*}
	&\quad \frac{d^2 \left(\log(H(t))- 2 \widetilde{M} (t) + t^2 (a^2 + b^2) (1 + M_1^2 + M_2^2) \right)}{dt^2} =\frac{d^2 \log\Big( H(t) e^{-2\widetilde{M}(t) + t^2 (a^2 + b^2) (1 + M_1^2 + M_2^2) }\Big)}{dt^2}\geq 0,
	\end{align*}
	meaning that $\log\Big( H(t) e^{-2\widetilde{M}(t) + t^2 (a^2 + b^2) (1 + M_1^2 + M_2^2) }\Big)$ is convex in $t$, that is,
	\begin{equation*}
		H(t) e^{-2\widetilde{M}(t) + t^2 (a^2 + b^2) (1 + M_1^2 + M_2^2) }
		\leq
		H(0)^{1-t} \left(H(1) e^{-2\widetilde{M}(1) + (a^2 + b^2) (1 + M_1^2 + M_2^2) }\right)^t
	\end{equation*}
	From \eqref{eqn-220611-0528} and \eqref{eqn-220705-0739},
	\begin{equation*}
		|\widetilde{M}(t)| \leq \sqrt{a^2 + b^2}(M_1 + M_2) t.
	\end{equation*}
	Eventually,
	\begin{equation*}
	\begin{split}
	H(t) 
	&\leq 
	H(0)^{1-t} H(1)^t e^{2\widetilde{M}(t) - t^2 (a^2 + b^2) (1 + M_1^2 + M_2^2) + t \left( -2\widetilde{M}(1) + (a^2 + b^2) (1 + M_1^2 + M_2^2) \right)}
	\\&\leq
	C e^{(a^2 + b^2) (1 + M_1^2 + M_2^2)} H(0)^{1-t} H(1)^t.
	\end{split}
	\end{equation*}
	
	Hence, \eqref{eqn-220705-1232} is proved. We are left to bound the derivatives in \eqref{eqn-220521-0546}.
    From \eqref{eqn-220609-1118}, \eqref{eqn-220611-0601}, \eqref{eqn-220515-1028-1}, the Cauchy-Schwarz inequality, and \eqref{eqn-220705-0739}, we have
	\begin{align*}
	&\frac{d}{dt}\left(\frac{dH(t)}{dt} - 2 \mathrm{Re} \left< (a + ib)( V f + h), f \right> \right) \\
	&\quad= 
	2\frac{dD(t)}{dt}
	\\&\quad=
	2 \left((a^2 + b^2) \left< [\cS, \cA] f, f \right> + 2 \| (a\cS + ib\cA)f \|_{L_x^2}^2 + 2 \mathrm{Re} \left< (a\cS + ib\cA) f, (a + ib)( V f + h)\right>\right)
	\\&\quad\geq
	2 (a^2 + b^2) \left(5\beta\lambda^2\int_{\bR^n}|\nabla f(t,x)|^2\,dx + 30 \beta^3\lambda^2\int_{\bR^n}|xf(t,x)|^2\,dx - H(t)\right)
	\\&\qquad  +
	4 \| (a\cS + ib\cA)f \|_{L_x^2}^2 - 4 \| (a\cS + ib\cA)f \|_{L_x^2} \cdot \|(a + ib)( V f + h)\|_{L_x^2}
	\\&\quad \geq
	2 (a^2 + b^2) \left(5\beta\lambda^2\int_{\bR^n}|\nabla f(t,x)|^2\,dx + 30 \beta^3\lambda^2\int_{\bR^n}|xf(t,x)|^2\,dx - H(t)\right) - \|(a + ib)( V f + h)\|_{L_x^2}^2
	\\&\quad\geq
	2 (a^2 + b^2) \left( 5  \beta\lambda^2\int_{\bR^n}|\nabla f(t,x)|^2\,dx + 30 \beta^3\lambda^2\int_{\bR^n}|xf(t,x)|^2\,dx - (1 + M_1^2 + M_2^2) H(t) \right).
	\end{align*}
	Multiplying $t(1-t)$ to both sides, integrating in time over $[0,1]$, then doing integration by parts, we obtain
	\begin{align}
	&2 (a^2 + b^2) \big( 5\beta\lambda^2\int_0^1 \int_{\bR^n}t(1-t)|\nabla f(t,x)|^2\,dxdt + 30 \beta^3\lambda^2\int_0^1\int_{\bR^n}t(1-t)|xf(t,x)|^2\,dxdt 
	\\ & \quad - (1 + M_1^2 + M_2^2) \int_0^t t(1-t)H(t)\,dt \big) 
	\\ &\leq
	\int_0^1 t(1-t)\frac{d^2H(t)}{dt^2}\,dt - 2 \mathrm{Re} \int_0^1 t(1-t)\frac{d}{dt} \left<(a + ib)( V f + h) , f \right>_{L^2_x}\,dt
	\\ &=
	\cancel{t(1-t)\frac{dH}{dt}\bigg\vert_0^1} - \int_0^1 (1-2t)\frac{dH}{dt}\,dt - \cancel{2 \mathrm{Re} \left( t(1-t) \left< (a + ib)( V f + h) , f \right>_{L^2_x}\bigg\vert_0^1 \right)} 
	\\&\quad+ 2 \mathrm{Re} \int_0^1 (1-2t)\left< (a + ib)( V f + h) , f \right>_{L^2_x} \, dt 
	\\\quad&\leq 
	-(1-2t)H(t)\big\vert_0^1 - 2\int_0^1 H(t)\,dt + 2 \sqrt{ a^2 + b^2 } ( M_1 + M_2 ) \int_0^1 H(t) \, dt 
	\\&=
	H(1) + H(0) + 2\left( 2 \sqrt{ a^2 + b^2 } ( M_1 + M_2 ) - 1 \right)\int_0^1 H(t) \, dt.\label{eqn-220521-0509}
	\end{align}
	Furthermore, from \eqref{eqn-220705-1232},
	\begin{equation}\label{eqn220521-0508-2}
	H(t) \leq C e^{(a^2 + b^2) (1 + M_1^2 + M_2^2)} \max\{H(0),H(1)\},\quad \forall t\in(0,1).
	\end{equation}
	Substituting \eqref{eqn220521-0508-2} and the following point-wise inequality
	\begin{equation}\label{eqn220521-0508-1}
	\beta|\nabla f|^2 + \beta^3|xf|^2 = \beta|\nabla (e^{\beta|x|^2}u)|^2 + \beta^3|xe^{\beta|x|^2}u|^2 \geq \frac{\beta}{4}|e^{\beta|x|^2}\nabla u| + \frac{\beta^3}{4}|e^{\beta|x|^2}xu|
	\end{equation}
	back to \eqref{eqn-220521-0509}, we reach \eqref{eqn-220521-0546}.
	
	This completes the proof of Lemma \ref{lem-220711}. 
\end{proof}

\subsubsection{Decay and regularity of solutions to \eqref{eqn-220705-0804}}\label{subsub-220711-2}
In this section, we first study the evolution of weighted $L^2_x$-norms. In particular, in the dissipative case ($a>0$ in Lemma \ref{lem-220705-1134}), solutions with Gaussian decay intial data will also have (possibly slower) Gaussian decay at later times. Actually our proof for Lemma \ref{lem-220705-1134} also works for a space-time dependent operator $\cL$, i.e., with $a_{kj}=a_{kj}(t,x)$.
\begin{lemma}\label{lem-220705-1134}
	Let $u \in L^\infty([0,1], L^2(\mathbb{R}^n)) \cap L^2([0,1],H^1(\mathbb{R}^n))$ solve \eqref{eqn-220705-0804}
	with
	\begin{equation*}
	 a^2 + b^2 \neq 0, \quad a \geq 0,\quad \text{and} \quad V = V(t , x) \in L^\infty([0,1]\times \mathbb{R}^n) \,\,\text{not necessarily being real-valued}.
	\end{equation*}
	Furthermore, assume that for some $\gamma>0$,
	\begin{equation*}
		e^{\gamma |x|^2} u(0) \in L^2(\mathbb{R}^n) , \quad e^{\gamma |x|^2} g(t , x) \in L^\infty([0,1],L^2(\mathbb{R}^n)).
	\end{equation*}
	Then, for any $t\in[0,1]$, $u(t)$ has a Gaussian decay:
	\begin{align}
	 \| e^{\frac{\gamma \lambda a \abs{x}^2}{\lambda a + 4 \gamma (\lambda a^2 \Lambda + 4b^2 C \norm{A}_{L^\infty}^2 ) t}} u(t) \|_{L^2_x} & \leq e^{ t \|a \mathrm{Re} V - b \mathrm{Im} V\|_{L^\infty_{t,x}} } \| e^{\gamma |x|^2} u(0) \|_{L^2_x} \\
	 & \quad + \sqrt{a^2 + b^2} \int_0^t e^{ (t -s) \|a \mathrm{Re} V - b \mathrm{Im} V\|_{L^\infty_{t,x}} } \|e^{\frac{\gamma \lambda a \abs{x}^2}{\lambda a + 4 \gamma (\lambda a^2 \Lambda + 4b^2 C \norm{A}_{L^\infty}^2 ) s}} g(s)\|_{L^2_x} \, ds   ,
	\end{align}
	where $C$ is a constant depending only on the spatial dimension $n$.
\end{lemma}
\begin{proof}[Proof of Lemma \ref{lem-220705-1134}]
	Let
	\begin{equation*}
		v := e^\phi u\quad\text{with}\quad\phi=\phi(t,x)
	\end{equation*}
	being a weight function to be chosen later. In the following, all the integration by parts are symbolic and will be justified once we fix $\phi$. In the proof, we write
	\begin{equation*}
		\left<\cdot,\cdot\right> = \left<\cdot,\cdot\right>_{L^2_x}, \quad \text{and} \quad e^\phi \cL e^{- \phi} = \cS + \cA
	\end{equation*}
	by the decomposition into symmetric($\cS$) and antisymmetric($\cA)$ parts with respect to $\left<\cdot,\cdot\right>_{L^2_x}$. 
    Then $v$ satisfies
	\begin{align}
	\p_t v & = (a + ib) e^\phi \cL (e^{-\phi} v) + (a +ib) e^{\phi} g + (a + ib) V v + \partial_t \phi v\\
	& = (a + ib) (\mathcal{S} + \mathcal{A}) v + (a +ib) e^{\phi} g + (a+ ib) (\mathrm{Re} V + i \mathrm{Im} V) v + \partial_t \phi v \\
	& = : \widetilde{S} v + \widetilde{A} v + (a + ib) e^{\phi} g 
	\end{align}
	where 
	\begin{align*}
		\widetilde{S} & = (a \mathcal{S} + ib \mathcal{A}) + (a \mathrm{Re} V - b \mathrm{Im} V) + \p_t \phi,\\
		\widetilde{A} & = (a \cA + ib \cS) + (a\mathrm{Im} V + b \mathrm{Re} V).
	\end{align*}
	Then, 
	\begin{align}
	\partial_t \| v \|_{L_x^2}^2 & = 2 \mathrm{Re} \inner{\widetilde{\mathcal{S}}v, v} + 2 \mathrm{Re} \inner{(a+ib)e^{\phi}g , v} . \label{eqn-220701-0100-1}
	\end{align}
	Now we bound
	\begin{equation*}
		\mathrm{Re} \inner{\widetilde{\mathcal{S}}v, v} = \mathrm{Re} \inner{(a \mathcal{S} + ib \mathcal{A})v , v} + \mathrm{Re} \inner{(a \mathrm{Re} V - b \mathrm{Im} V) v,v} + \mathrm{Re} \inner{(\partial_t \phi)v, v}
	\end{equation*}
term by term. First, we consider the inner product involving $\mathcal{S}$ . By doing integration by parts, we write 
    \begin{align}
	\mathrm{Re} \inner{\mathcal{S} v,v} & =  \inner{\mathcal{S}v,v } =  \int \partial_k (a_{kj} \partial_j v) \overline{v} \, dx + \int \partial_j \phi \partial_j \phi a_{kj} v \overline{v} \, dx \\
	& = - \int a_{kj} (\partial_j v) (\partial_j \overline{v}) \, dx + \int a_{kj} (\partial_k \phi ) (\partial_j \phi )v \overline{v} \, dx. \label{eqn-221021-0756-1}
	\end{align}
For the $\mathcal{A}$ term, similarly, we have
 	\begin{align}
	\mathrm{Re} \inner{i \mathcal{A} v,v} & = \mathrm{Im} \inner{ \mathcal{A} v,v} \\
	& = - \mathrm{Im}  \int (\partial_j \phi ) \partial_k (a_{kj} v) \overline{v} \, dx  - \mathrm{Im} \int  (\partial_k \phi ) a_{kj} (\partial_j v)  \overline{v} \, dx  - \cancel{\mathrm{Im}  \int (\partial_{kj}^2 \phi) a_{kj} v \overline{v} \, dx } \\
	& = \cancel{ \mathrm{Im}  \int a_{kj} (\partial_{kj}^2 \phi) v \overline{v} \, dx } + \mathrm{Im}  \int a_{kj} (\partial_j \phi ) v (\partial_k \overline{v}) \, dx  - \mathrm{Im}  \int (\partial_k \phi) a_{kj} (\partial_j v) \overline{v} \, dx\\
	& = 2 \mathrm{Im} \int (\partial_k \phi) a_{kj} (\partial_j v) \overline{v} \, dx , \label{eqn-221021-0756-2}
	\end{align}
	where in the last line we used 
	\begin{align}
	\left<A (\nabla \phi) v , \nabla v \right> - \left<\nabla v , A (\nabla \phi) v \right> = 2 i \mathrm{Im}\left<A \nabla \phi v , v \right>.
	\end{align}
	Combining \eqref{eqn-221021-0756-1}-\eqref{eqn-221021-0756-1}, and noting
	\begin{align}
	\lambda \leq (a_{kj}) \leq \Lambda,   
	\end{align}
	we have	
	\begin{align}
	\mathrm{Re} \inner{\widetilde{\mathcal{S}}v, v } & = \mathrm{Re}  \inner{ (a \mathcal{S} + ib \mathcal{A})v, v } + \mathrm{Re} \inner{ (a \mathrm{Re}  V - b \mathrm{Im} V)v,v} + \mathrm{Re}  \inner{\partial_k \phi v,v} \\
	& = -a \int a_{kj} (\partial_j v) (\partial_k \overline{v}) \, dx + a \int a_{kj} (\partial_k \phi) (\partial_j \phi) v \overline{v} \, dx \\
	& \quad + 2 b \mathrm{Im} \int (\partial_k \phi) a_{kj} (\partial_j v) \overline{v} \, dx + \int (a \mathrm{Re} V- b \mathrm{Im} V + \partial_k \phi) \abs{v}^2 \, dx \\
	& \leq -a \lambda \|\nabla v\|_{L^2_x}^2 + a \Lambda \|v \nabla \phi\|_{L^2_x}^2 + 2 b \mathrm{Im} \int (\partial_k \phi) a_{kj} (\partial_j v) \overline{v} \, dx + \int (a \mathrm{Re} V- b \mathrm{Im} V + \partial_k \phi) \abs{v}^2 \, dx.
	\end{align}
Furthermore,
	\begin{align}
	2 b \mathrm{Im} \int (\partial_k \phi) a_{kj} (\partial_j v) \overline{v} \, dx  
	& \leq  
	C  | b |  \Lambda^2 \int \frac{\sqrt{\lambda a}}{\sqrt{\lambda a}} \abs{\nabla \phi} \abs{\nabla v} \abs{v} \, dx \\
	& \leq \int \lambda a \abs{\nabla v}^2 + C \frac{b^2}{\lambda a} \Lambda^2 \abs{\nabla \phi}^2 \abs{v}^2 \, dx \\
	& = \lambda a \norm{\nabla v}_{L^2_x}^2 + C \frac{b^2}{\lambda a} \Lambda^2 \int \abs{\nabla \phi}^2 \abs{v}^2 \, dx
	\end{align}
	where $C=C(n)$.  
    Hence,
	\begin{align}
	\mathrm{Re} \inner{\widetilde{\mathcal{S}}v,v} & \leq \int (a \mathrm{Re} V - b \mathrm{Im}V + \partial_t \phi ) \abs{v}^2 \, dx + a \Lambda \int \abs{\nabla \phi}^2 \abs{v}^2 \, dx + C \frac{b^2}{\lambda a} \Lambda^2 \int \abs{\nabla \phi}^2 \abs{v}^2 \, dx \\
	& \leq \norm{a \mathrm{Re} V - b \mathrm{Im} V }_{L^\infty} \norm{v}_{L_x^2}^2 + \int [(a \Lambda) \abs{\nabla \phi}^2 + C \frac{b^2}{\lambda a} \Lambda^2 \abs{\nabla \phi}^2 + \partial_t \phi ] \abs{v}^2 \, dx . \label{eqn-220701-0100-2}
	\end{align}
	We are going to choose a proper $\phi$ such that
	\begin{equation}\label{eqn-220701-1244}
		(a \Lambda) \abs{\nabla \phi}^2 + C \frac{b^2}{\lambda a}  \Lambda^2 \abs{\nabla \phi}^2 + \partial_t \phi \leq 0.
	\end{equation}
	For simplicity, we choose $\phi$ in the format of
	\begin{equation*}
		\phi = \alpha(t) \left( \left( \min\{ |x|,R \} \right)^2 \ast \eta_\epsi \right) ,\quad 
		\text{where} \,\, 
		R > \epsi >0\,\,
		\text{and} \,\,\eta_\epsi\in C^\infty (B_\epsi) \,\,\text{is a usual mollifier} .
	\end{equation*}
	With such choice, the requirement \eqref{eqn-220701-1244} becomes
	\begin{align}
	    ( 4a \Lambda + C\frac{4b^2}{\lambda a} \Lambda^2  ) \alpha^2 (t) + \alpha' (t) \leq 0 ,\quad  \alpha'(t)\leq 0,
	\end{align}
for which we solve
	\begin{align}
	\begin{cases}
	\alpha'(t) = - 4 (a \Lambda + C \frac{4b^2}{\lambda a} \norm{A}_{L^\infty}^2) \alpha^2 (t) ,\\
	\alpha(0) = \gamma .
	\end{cases}
	\end{align}
	This yields
	\begin{align}
	\alpha(t)= \frac{\gamma \lambda a}{\lambda a + 4 \gamma (\lambda a^2 \Lambda + 4b^2 \norm{A}_{L^\infty}^2 C) t}.
	\end{align}
    Combining \eqref{eqn-220701-0100-1}, \eqref{eqn-220701-0100-2}, and \eqref{eqn-220701-1244}, we have
	\begin{align}
	\partial_t \norm{v}_{L^2_x}^2 \leq 2 \norm{a \mathrm{Re} V - b \mathrm{Im} V}_{L^\infty_{t,x}} \norm{v(t)}_{L^2_x}^2 + 2 \sqrt{a^2 + b^2} \norm{e^{\phi} g}_{L^2_x} \norm{v(t)}_{L^2_x},
	\end{align}
	from which,
	\begin{equation*}
		\p_t \norm{v}_{L^2_x} \leq \norm{a \mathrm{Re} V - b \mathrm{Im} V}_{L^\infty_{t,x}} \norm{v(t)}_{L^2_x} +  \sqrt{a^2 + b^2} \norm{e^{\phi} g}_{L^2_x}.
	\end{equation*}
	Hence,
	\begin{equation*}
		\norm{v(t)}_{L^2_x} \leq e^{t \norm{a \mathrm{Re} V - b \mathrm{Im} V}_{L_{t,x}^\infty} }\norm{v(0)}_{L^2_x} + \sqrt{a^2 + b^2} \int_0^t e^{ (t - s) \norm{a \mathrm{Re} V - b \mathrm{Im} V}_{L^\infty_{t,x}} }\norm{e^{\phi(s)} g(s)}_{L^2_x} \, ds.
	\end{equation*}
	Passing $\epsi\rightarrow 0, R\rightarrow \infty$, we finish the proof of Lemma \ref{lem-220705-1134}.
\end{proof}
Combining Lemma \ref{lem-220705-1134}, semigroup theories, and standard parabolic regularity theories, we have the following regularity result.
The proof of Lemma \ref{lem-220706-1211} can be found in Appendix \ref{sec AppC}.
\begin{lemma}\label{lem-220706-1211}
	Suppose $u \in C([0,1],L^2(\mathbb{R}^n))$ is a solution to \eqref{eqn-220630-1118}  with
	\begin{equation*}
	    a>0,\, b\in \bR, V=V(x) \,\,\text{being real-valued and bounded}.
	\end{equation*}
	If for some $\beta>0$ 
	\begin{equation} \label{eqn-220706-1234}
		e^{\beta |x|^2} u(t,x) \in L^2(\mathbb{R}^n)\quad \forall t \in (0,1),
	\end{equation}
	then we have $u \in L^2_{t,loc}( (  0,1  ) , H^2_{x,loc}(\mathbb{R}^n))$ and for all $\epsi>0$,
	\begin{equation*}
		 |u(t,x)| + |\nabla_x u(t,x)|= o ( e^{- (\beta - \epsi) |x|^2} ) \quad \text{as} \,\, |x|\rightarrow \infty, \quad \forall t\in (0,1).
	\end{equation*}
\end{lemma}


\subsubsection{$Log$-convexity property with dissipation (Proposition \ref{prop-220704-0508})}\label{subsub-220711-3}

\begin{proof}[Proof of Proposition \ref{prop-220704-0508}]
First, applying Lemma \ref{lem-220705-1134}, we obtain a sub-Gaussian decay $e^{\beta |x|^{2-\delta}} u(t,x) \in L^2(\mathbb{R}^n)$ for all $t\in(0,1)$ and all small $\delta>0$.
From this, we can employ the integartion by parts and prove a $log$-convexity result similar to \eqref{eqn-220705-1232}, but with $e^{ \beta |x|^2}$ being replaced by $ e^{\beta |x|^{2-2\delta}} $, which in particular, implies
	\begin{equation}\label{eqn-220908-0900}
		\sup_{ t \in (0,1) } \|e^{\beta |x|^{2 - 2\delta}} u(t,x)\|_{L^2_x} \leq C \quad \text{with $C$ independent of} \,\, \delta.
	\end{equation}
	The proof is very similar to that of Lemma \ref{lem-220711}, which we sketch below. For a usual cutoff function $\eta = \eta(x) \in C^\infty_c (B_1^{\mathsf{c}})$ with $\eta = 1$ on $B_2^{\mathsf{c}}$, let
	\begin{equation*}
		f^\delta = e^{\beta |x|^{2 - \delta}} u \eta.
	\end{equation*}
	Then $f^\delta$ solves
	\begin{equation*}
		\p_t f^\delta = (a + ib) (e^{ \beta |x|^{2 - \delta}} (\cL + V) e^{ - \beta |x|^{2 - \delta}} f^\delta + h^\delta) ,
	\end{equation*}
	where
	\begin{equation*}
		h^\delta = - e^{\beta |x|^{2 - \delta}} [\cL , \eta] u \in C^\infty_c ( [0,1] \times (B_2 \setminus B_1) ).
	\end{equation*}
	Define $\cS^{\delta}$ and $\cA^{\delta}$ from the decomposition
	\begin{equation*}
		e^{\beta |x|^{2-\delta}} \cL e^{-\beta |x|^{2-\delta}} = \cS^{\delta} + \cA^{\delta}
	\end{equation*}
	into symmetric and antisymmetric parts. Direct computation shows
	\begin{align*}
		& \quad \left< [\cS^\delta, \cA^\delta] f , f \right>_{L^2_{x}}
		\\&= 
		4 \beta (2-\delta) \int a_{kj} a_{mk} |x|^{-\delta} \p_j f \p_m f \, dx+ \beta \delta O(1) \int |x|^{-\delta} |\nabla f|^2 \, dx+ \beta \int |x| |\nabla A| |x|^{-\delta} |\nabla f|^2 \, dx
		\\&\quad +
		O(1) \beta \int \left( |x|^{-\delta - 1} + |x|^{-\delta} |\nabla A| + |x|^{1-\delta} |\nabla^2 A| + |x|^{1-\delta} |\nabla A|^2 \right) |f| |\nabla f|\, dx
		\\&\quad +
		4 \beta^3 (2-\delta)^3 \int |x|^{-3 \delta} x_l x_k a_{jl} a_{kj} |f|^2 \, dx + O(1) \beta^3 \delta \int |x|^{2 - 3a} |f|^2\, dx + O(1) \beta^3 \int |x|^{3-3a}	|\nabla A| |f|^2\, dx
		\\&\quad + O(1) \beta \int \left( |x|^{-\delta - 1} |\nabla A| + |x|^{-\delta} |\nabla^2 A| + |x|^{1-\delta} |\nabla^3 A| + |x|^{-\delta - 2} + |x|^{1 - \delta} |\nabla A| |\nabla^2 A| + |x|^{-\delta}|\nabla A|^2 \right) |f|^2\, dx .
 	\end{align*}
 	Further following the computation \eqref{eqn-220508-0350-1} - \eqref{eqn-220515-1012} and noticing $f^\delta = 0$ when $|x| \leq 1$, such commutator term is nonnegative if we choose $\delta$ and $\epsi_0$ in \eqref{eqn-220512-1132} to be sufficiently small and $\beta$ sufficiently large. The rest of the proof follows from similar steps in Lemma \ref{lem-220711}. Here, the Gaussian decay in Lemma \ref{lem-220705-1134} together with the reguarity and derivative decay in Lemma \ref{lem-220706-1211} guarantee that all the integration by parts are justified. Hence, \eqref{eqn-220908-0900} is proved.
	
	From \eqref{eqn-220908-0900} and the monotone convergence theorem, we know 
	\begin{equation*}
	e^{\beta |x|^2 } u(t,x) \in L_x^2(\mathbb{R}^n) \quad \forall t \in (0,1).
	\end{equation*}
	Using Lemma \ref{lem-220706-1211}, we know $u \in L^2_{t,loc} H^2_{x,loc}$ and for all $\epsi > 0$,
	\begin{equation*}
		|u(t,x)| + |\nabla_x u(t,x)|= o ( e^{- (\beta - \epsi) |x|^2} ).
	\end{equation*}
	Hence, we can apply Lemma \ref{lem-220711} with $\beta$ being replaced by $\beta - 2 \epsi$. Finally, Proposition \ref{prop-220704-0508} follows from passing $\epsi$ to zero and another application of the monotone convergence theorem.
\end{proof}

\subsection{Proof of Theorem \ref{thm-220515-0921} and Corollary \ref{cor-220515-1117}} \label{ssec 3.2}

\begin{proof}[Proofs of Theorem \ref{thm-220515-0921}]
	Define
	\begin{equation}\label{eq u_e}
		u_\epsi (t) := e^{t \epsi (\cL + V)} u(t) ,
	\end{equation}
	to be the parabolic flow approximation of $u$. From the semigroup property
\begin{equation} \label{eqn-220706-1230}
e^{(z_1 + z_2)(\cL + V)} = e^{z_1 (\cL + V)} e^{z_2 (\cL + V)}, \quad \forall \, \mathrm{Re} z_1, \mathrm{Re} z_2 \geq 0,
\end{equation}
we have
\begin{equation}
    u_\epsi (t) = e^{t (\epsi + i) (\cL + V)} u_0,
\end{equation}
i.e., $u_\epsi$ solves
	\begin{equation*}
	\begin{cases}
	\p_t u_\epsi = (\epsi + i) (\cL + V)u_\epsi ,\\
	u_\epsi (0) = u_0 .
	\end{cases}.
	\end{equation*}
	Here we used the fact that $V=V(x)$ is real-valued.
	Furthermore, from
	\begin{equation*}
		u_\epsi (1) = e^{\epsi (\cL + V)} u(1),
	\end{equation*}
	we can apply Lemma \ref{lem-220705-1134} to obtain that for some $\beta_\epsi \in (0, \beta)$ with $\beta_\epsi \rightarrow \beta$ as $\epsi \rightarrow 0$,
	\begin{equation*}
		e^{\beta_\epsi |x|^2} u_\epsi (1,x) \in L_x^2(\mathbb{R}^n).
	\end{equation*}
	Now we apply Proposition \ref{prop-220704-0508} with $a=\epsi, b=1, g=0, \beta=\beta_\epsi$ to $u_\epsi$ to obtain \eqref{eqn-220705-1232} - \eqref{eqn-220521-0546} with $\beta , u$ being replaced by $\beta_\epsi, u_\epsi$.
	The desired estimate follows by sending $\epsi \rightarrow 0$.
	
	This completes the proofs of Theorem \ref{thm-220515-0921}. 
\end{proof}

\begin{proof}[Proof of Corollary \ref{cor-220515-1117}]
The proof is based on a ``subordination type'' inequality in Lemma \ref{lem1.Appendix} in Appendix \ref{sec AppA}. 
We only prove the inequality \eqref{cor.high.order.decay.1} since \eqref{cor.high.order.decay.2} can be derived analogously by noting that $\beta_0 \geq 1$.

We start by proving \eqref{cor.high.order.decay.1} in the case $\alpha\in (1,2)$. 
Taking $q \in (1,2)$ as the conjugate exponent of $\alpha$, multiplying both members of \eqref{eqn-220711-0221} by $e^{-\frac{(2\beta)^q}{q\kappa^q}}(2\beta)^\frac{q-2}{2}$ with $\kappa > 0$, and integrating against $\beta$ in $(\beta_0,\infty)$, we get
\begin{align}
     & \quad \int_{\mathbb{R}^n}\left(\int_{\beta_0}^{+\infty} e^{2\beta|x|^2-\frac{(2\beta)^q}{\kappa^q q}}(2\beta)^\frac{q-2}{2} d\beta\right)|u(t,x)|^2 \, dx \label{pf.cor.log.conv.eq1}\\
    & \leq C\left(\int_{\mathbb{R}^n} \left(\int_{\beta_0}^{+\infty} e^{2\beta|x|^2-\frac{(2\beta)^q}{\kappa^q q}}(2\beta)^\frac{q-2}{2} d\beta\right)|u_0(x)|^2 \, dx\right)^{1-t} \left(\int_{\mathbb{R}^n} \left(\int_{\beta_0}^{+\infty} e^{2\beta|x|^2-\frac{(2\beta)^q}{\kappa^q q}}(2\beta)^\frac{q-2}{2} d\beta\right)|u_1(x)|^2 \, dx\right)^t .       
\end{align}
By Lemma \ref{lem1.Appendix} with $r= |x|^2$ and $\alpha=p$, for every $\kappa$ with 
\begin{equation}
    \frac{\kappa^\alpha}{\alpha} >\kappa_0:=\frac{1}{\alpha}\left( 4\beta_0 \left(\frac{2}{q-2} \right)^{1/q}\right)^\alpha,
\end{equation}
we have
\begin{align}
\int_{\beta_0}^{+\infty} e^{2\beta|x|^2-\frac{(2\beta)^q}{\kappa^q q}}(2\beta)^\frac{q-2}{2} d\beta\approx_{\kappa, \alpha} e^{\frac{\kappa^\alpha  |x|^{2\alpha}}{\alpha}},
\end{align}
which, replaced in \eqref{pf.cor.log.conv.eq1}, gives the result for $\alpha\in (1,2)$. Here  the notation $f\approx_{a,b} g$ means that there exists a positive constant $C$ depending on $a,b$  such that $1/Cg(x)\leq f(x)\leq Cg(x)$.

The proof of \eqref{cor.high.order.decay.1} in the case when $\alpha=2^m$, with $m$ being an integer $m\geq 1$, follows by using the strategy of ``completing the square''. We just show how to prove the case $m=1$, since iterating this result one can further get the conclusion for any $m\geq 1$.

Multiplying $e^{-\left(\frac{\beta}{\kappa}\right)^2}$ to both sides of \eqref{eqn-220711-0221}, with $\kappa \geq \beta_0$ (recall, $\beta_0$ is as before such that \eqref{eqn-220711-0222} holds for $\beta\geq \beta_0$), and completing the square, we obtain
	\begin{align*}
		& \quad \int_{\bR^n} e^{-(\beta/\kappa-\kappa|x|^2)^2}e^{\kappa^2|x|^4}|u(t,x)|^2\,dx \\
		& \leq 
		C\left(\int_{\bR^n} e^{-(\beta/\kappa-\kappa|x|^2)^2}e^{\kappa^2|x|^4}|u(0,x)|^2\,dx\right)^{1-t}
		\left(\int_{\bR^n} e^{-(\beta/\kappa-\kappa|x|^2)^2}e^{\kappa^2|x|^4}|u(1,x)|^2\,dx\right)^{t}.
	\end{align*}
	Next we integrate both sides in $\beta$, from $\beta_0$ to $+\infty$. By H\"older's inequality and Fubini's theorem,
	\begin{equation}\label{eqn-220611-0551}
		\begin{split}
		&\quad\int_{\bR^n}e^{\kappa^2|x|^4}|u(t,x)|^2\int_{\beta_0}^\infty e^{-(\beta/\kappa-\kappa|x|^2)^2}\,d\beta dx 
		\\&\leq 
		C
		\left(\int_{\bR^n}e^{\kappa^2|x|^4}|u(0,x)|^2\int_{\beta_0}^\infty e^{-(\beta/\kappa-\kappa|x|^2)^2}\,d\beta dx\right)^{1-t}
		\left(\int_{\bR^n}e^{\kappa^2|x|^4}|u(1,x)|^2\int_{\beta_0}^\infty e^{-(\beta/\kappa-\kappa|x|^2)^2}\,d\beta dx\right)^{t}.
	\end{split}
	\end{equation}
	Taking the change of variable $\gamma=\beta/\kappa - \kappa|x|^2$, since $\kappa\geq  \beta_0$
	\begin{equation}\label{eqn-220521-0558}
		\int_{\beta_0}^\infty e^{-(\beta/\kappa-\kappa|x|^2)^2}\,d\beta = \kappa \int_{\beta_0/\kappa - \kappa |x|^2}^\infty e^{-\gamma^2}\,d\gamma
		\,\,
		\begin{cases}
			\leq \kappa \int_{-\infty}^\infty e^{-\gamma^2} = \kappa\sqrt{\pi}
			\\
			\geq \kappa \int_1^\infty e^{-\gamma^2} > \kappa/10
		\end{cases}.
	\end{equation}
	Substituting the upper (lower) bound in \eqref{eqn-220521-0558} back to the right (left)-hand side of \eqref{eqn-220611-0551}, we reach the desired estimate.
	
	We are now left with the case $\alpha\in(2^m, 2^{m+1})$ with $m\geq 1$. The strategy to be used is the same as the one used for the case $\alpha\in (1,2)$. The only difference here is that Lemma \ref{lem1.Appendix} is applied with $r=|x|^{ 2^{m+1}}$ and $p=\alpha/2^m \in (1,2)$. Corollary \ref{cor-220515-1117} is proved.
\end{proof}

\section{Carleman inequalities and lower bounds of solutions in the general case}\label{sec Carlman}

In this section, we prove a Carleman estimate with a cubic exponential weight (Lemma \ref{lem-220816-1207}), and then show a lower bound for non-trivial solutions (Theorem \ref{thm-220815-1155}). The latter will be used in Section \ref{sec Main} to prove the main theorem (Theorem \ref{thm-220816-1210}).

\begin{lemma}\label{lem-220816-1207}
    Assume $A \in C_b^3(\bR^n)$, $\varphi = \varphi(t) \in C^\infty_c(\bR)$, and $r_0>0$. There exists an $\epsi_0=\epsi_0 (n,\lambda, \Lambda)$, such that if
    \begin{equation*}
        \sup_{x\in\bR^n} |x||\nabla A| < \epsi_0,
    \end{equation*}
    then for any function
    \begin{equation}
        f  \in C^\infty_c (\bR\times B_{r_0}^{\mathsf{c}})
    \end{equation}
    and $\beta ,R$ satisfying
    \begin{equation} \label{eqn-221005-0106}
        \beta \geq \beta_1 := \max\{ \lambda^{-1} \|\varphi ''\|_{L^\infty}^{1/2} r_0^{-1}R^3 , C_1( 1+ r_0^{-1}) R^2\}, \quad R\geq 1
    \end{equation}
    with $C_1 = C_1 (n,\lambda, \Lambda, \|A\|_{C^3})$, we have
    \begin{align}
         \beta R^{-2} \int_\bR \int_{\bR^n} |\nabla f|^2 \, dxdt & + \beta^3 R^{-6}\int_\bR \int_{\bR^n} |x|^2 |f|^2\, dxdt \\
        & \leq
        \lambda^{-2} \int_\bR \int_{\bR^n} \big| e^{\beta (|x/R|^2 + \varphi(t))} (i\p_t + \cL) e^{- \beta (|x/R|^2 + \varphi(t))} f \big|^2\, dxdt
    \end{align}
\end{lemma}

\begin{proof}[Proof of Lemma \ref{lem-220816-1207}]
Let $\phi = \phi (t,x)$.  For $\cL = \partial_{k} (a_{kj} \partial_{j} )$ with $a_{kj} = a_{kj} (x)$, we decompose 
\begin{align}
e^{\phi} (i \partial_{t} + \cL) e^{-\phi} =:  \cS + \cA,
\end{align}
where $\cS$ and $\cA$ are symmetric and anti-symmetric operators with respect to $\inner{ \cdot, \cdot}_{L_{t,x}^2}$. 

Now we plug in $\phi = \beta (\abs{x/R}^2 + \varphi(t))$. Using the formula in Section \ref{sub-220817}, noting
\begin{equation*}
\partial_{j} \phi  = \frac{2 \beta}{R^2} x_j , \,\, \partial_{kj}^2 \phi  = \frac{2\beta}{R^2} \delta_{kj}, \,\, \partial_{t} \phi  = \beta \varphi' , \,\, \partial_{tt} \phi  = \beta \varphi'',
\end{equation*}
and
\begin{equation*}
    \nabla^3 \phi = \nabla \partial_{t} \phi = 0,
\end{equation*}
we have
\begin{align}
[\cS, \cA] & = - \frac{8\beta}{R^2} a_{kj} a_{mk} \partial_{mj}^2 - \frac{8 \beta}{R^2} a_{kj} (\partial_{k} a_{ml}) x_l \partial_{mj}^2 + \frac{4 \beta}{R^2} a_{ml} (\partial_{m} a_{kj}) x_l \partial_{kj}^2 \\
& \quad + O(1) \frac{\beta}{R^2} (\abs{\nabla A} + \abs{x} \abs{\nabla^2 A} + \abs{x} \abs{\nabla A}^2) \nabla \\
& \quad + 32 \frac{\beta^3}{R^6} a_{kj} a_{kl} x_l x_j + 16 \frac{\beta^3}{R^6} a_{ml} (\partial_{m} a_{kj}) x_k x_j x_l \\
& \quad + \beta \varphi''  + O(1) \frac{\beta}{R^2} (\abs{\nabla^2 A} + \abs{x} \abs{\nabla A} \abs{\nabla^2 A} + \abs{x} \abs{\nabla^3 A} + \abs{\nabla A}^2).
\end{align}
Hence,
\begin{align}
 \| e^{\phi} & ( i \p_t + \cL ) e^{-\phi} f \|_{L^2_{t,x}}^2
 =
\| (\cS + \cA) f \|_{L^2_{t,x}}^2
\\& \geq 
\inner{[\cS , \cA]f, f}_{L_{t,x}^2} 
\\ & = \frac{8\beta}{R^2} \iint a_{kj} a_{mk} \partial_{m} f \partial_{j} \bar{f}\, dxdt + O (1) \frac{\beta}{R^2} \iint \abs{x} \abs{\nabla A} \abs{\nabla f}^2\, dxdt \\
& \quad + O (1) \frac{\beta}{R^2} \iint (\abs{\nabla A} + \abs{x} \abs{\nabla A}^2 + \abs{x} \abs{\nabla^2 A}) \abs{\nabla f} \abs{f} \, dxdt\\
& \quad + 32 \frac{\beta^3}{R^6} \iint a_{kj} a_{kl} x_l x_k \abs{f}^2 \, dxdt + O (1) \frac{\beta^3}{R^6} \iint \abs{x}^3 \abs{\nabla A} \abs{f}^2\, dxdt + \iint \beta \varphi'' \abs{f}^2\, dxdt \\
& \quad + O (1) \frac{\beta}{R^2} \iint (\abs{\nabla^2 A}  + \abs{x} \abs{\nabla A} \abs{\nabla^2 A} + \abs{x} \abs{\nabla^3 A} + \abs{\nabla A}^2 + \abs{x}^2 \abs{\nabla^2 A}^2 + \abs{x}^2 \abs{\nabla A}^4) \abs{f}^2\, dxdt\\
& \geq
7\beta R^{-2} \lambda^2 \iint |\nabla f|^2\, dxdt + O (1) \beta R^{-2} \iint \abs{x} \abs{\nabla A} \abs{\nabla f}^2\, dxdt\\
& \quad + 32 \beta^3 R^{-6} \lambda^2 \iint |x|^2\abs{f}^2 \, dxdt + O (1) \beta^3 R^{-6} \iint \abs{x}^3 \abs{\nabla A} \abs{f}^2 \,dx dt+ \beta \iint \varphi'' \abs{f}^2 \, dxdt\\
& \quad + O (1) \beta R^{-2} \iint (\abs{\nabla^2 A}  + \abs{x} \abs{\nabla^3 A} + \abs{\nabla A}^2 + \abs{x}^2 \abs{\nabla^2 A}^2 + \abs{x}^2 \abs{\nabla A}^4) \abs{f}^2 \, dxdt .
\end{align}
Here we used
\begin{equation*}
    \iint \abs{x} \abs{\nabla A} \abs{\nabla^2 A} \abs{f}^2\, dxdt \leq \frac{1}{2} \iint \abs{x}^2 \abs{\nabla^2 A}^2 \abs{f}^2\, dxdt + \frac{1}{2} \iint \abs{\nabla A}^2 \abs{f}^2\, dxdt,
\end{equation*}
\begin{equation*}
    \iint (\abs{\nabla A} + \abs{x} \abs{\nabla A}^2 + \abs{x} \abs{\nabla^2 A}) \abs{\nabla f} \abs{f}\, dxdt \leq \delta \iint\abs{\nabla f}^2 + \frac{C}{\delta} \iint(\abs{\nabla A}^2 + \abs{x}^2 \abs{\nabla A}^4 + \abs{x}^2 \abs{\nabla^2 A}^2) \abs{f}^2\,dxdt,
\end{equation*}
and the following inequality coming from the ellipticity
\begin{equation*}
    a_{kj}a_{mk}\xi_m \xi_j \geq \lambda^2 |\xi|^2 \quad \forall \xi\in \bR^n,
\end{equation*}
then chose $\delta$ to be small enough.
Furthermore, choosing $\epsi_0$ in \eqref{eqn-220512-1132} to be small enough, we have
\begin{align}
\inner{[\cS , \cA]f, f}_{L_{t,x}^2} & \geq 4 \beta R^{-2} \lambda^2 \iint \abs{\nabla f}^2\,dx dt + 16 \beta^3 R^{-6} \lambda^2 \iint \abs{x}^2 \abs{f}^2\,dx dt \\
& \quad +\underbrace{ \beta \iint \varphi'' \abs{f}^2 \,dx dt + O (1) \beta R^{-2} \iint (\abs{\nabla^2 A} + \abs{x} \abs{\nabla^3 A} + \abs{\nabla A}^2 + \abs{x}^2 \abs{\nabla^2 A}^2) \abs{f}^2\,dx dt }_{I} .
\end{align}
Hence, if on $\supp (f)$
\begin{align}
8 \beta^3 R^{-6} \lambda^2 \abs{x}^2 \geq \beta \norm{\varphi''}_{L^{\infty}} + O (1) \beta R^{-2} (\abs{\nabla^2 A} + \abs{x} \abs{\nabla^3 A} + \abs{\nabla A}^2 + \abs{x}^2 \abs{\nabla^2 A}^2),
\end{align}
the term $I$ can be absorbed. Recall $\supp (f) \subset \{\abs{x} \geq r_0\}$, it suffices to assume
\begin{align}
\beta^2 R^{-6} \lambda^2 r_0^2 \geq \norm{\varphi''}_{L_x^{\infty}}
\end{align}
and
\begin{align}
\beta^2 R^{-4} \geq C \parenthese{\frac{\norm{\nabla^2 A}_{L_x^{\infty}}}{r_0^2} + \frac{\norm{\nabla^3 A}_{L_x^{\infty}}}{r_0}  + \frac{\norm{\nabla A}_{L_x^{\infty}}^2}{r_0^2}  + \norm{\nabla^2 A}_{L_x^{\infty}}^2 }
\end{align}
where $C =C(n , \lambda, \Lambda)$. Lemma \ref{lem-220816-1207} is proved.
\end{proof}

Next, we prove lower bounds for non-trivial solutions.
\begin{theorem}\label{thm-220815-1155}
  Let $u\in L^\infty ([0,1], L^2(\mathbb{R}^n)) \cap L^2_{loc}((0,1),H^1(\mathbb{R}^n))$ be a solution of
    \begin{equation*}
        \p_t  = i( \cL + V) u,
    \end{equation*}
    where $a_{jk}$ satisfies the assumptions of Lemma \ref{lem-220816-1207}, $V \in  L^\infty(\bR\times\bR^n,\bC)$. Denote by $M_1:=\|V\|_{L^\infty}$.
    
    Furthermore, let $E_1, E_2, R_0, \epsi$ be the numbers such that
    \begin{equation}\label{thm-220815-1155-eq0}
        \int_{ 1/8 }^{ 7/8 }\int_{\mathbb{R}^n}(|u|^2+|\nabla u|^2)(t,x) \, dt dx\leq E_1^2<\infty
    \end{equation}
    and
    \begin{equation} \label{thm-220815-1155-eq2}
        \int_{ 1/4 }^{ 3/4 }\int_{B_{R_0}\setminus B_{2\varepsilon}}|u|^2(t,x)\, dt dx\geq E_2^2,
      \end{equation}
    then there exist positive constants $R_1=R_1(n, \lambda , \Lambda,  \norm{A}_{C^3} , \varepsilon , M_1,E_1 , E_2 , R_0)$, $C_0=C_0(\lambda , \varepsilon)$,  and $C =C (n, \lambda , \Lambda, \norm{A}_{C^3} ,\varepsilon , M_1,E_1 , E_2 , R_0)$, such that, for any $R > R_1$,
    \begin{equation}\label{thm-220815-1155-eq1}
        \delta(R) := \int_{1/8}^{7/8} \int_{B_R\setminus B_{R-1}} (|u|^2 + |\nabla u|^2)(t,x) \,dxdt \geq C \,e^{-C_0R^3}.
    \end{equation}
\end{theorem}

\begin{proof}[Proof of Theorem \ref{thm-220815-1155}]
We start by defining a cutoff function in time 
\begin{align}
    \varphi = \varphi (t) \in C_c^\infty( (1/8, 7/8)), \quad \text{with} \,\, \varphi = 3 \,\,\text{on}\,\, [1/4,3/4].
\end{align}
Furthermore, define $\theta, \theta^{(R)}, \theta_{(\varepsilon)} \in C^{\infty}(\mathbb{R}^n)$:
\begin{align}
\theta(r)=\left\{\begin{array}{ll}
    0 &  r\leq 1  \\
    1 & r\geq 2 
\end{array}\right., 
\quad  
\theta^{(R)} (x)  = 
\begin{cases}
1, & x \in B_{R-1}, \\
0, & x \in B_R^{\mathsf{c}} . 
\end{cases}
\quad 
\theta_{(\varepsilon)} (x) = 
\begin{cases}
0, & x \in B_{\varepsilon}, \\
1, & x \in B_{2 \varepsilon}^{\mathsf{c}}  .
\end{cases}
\end{align}

As usual, we require
\begin{equation*}
    |\varphi'| + |\theta'| + |\nabla \theta^{(R)}| + |\nabla^2 \theta^{(R)}| \leq C,\quad |\theta_{(\epsi)}'|\leq \frac{C}{\epsi}.
\end{equation*}
Write also
\begin{align}
\psi(t,x) = \abs{x/R}^2 + \varphi (t) \left( = \frac{\phi (t,x)}{\beta} \right),
\end{align}
and let
\begin{align}
g = \theta (\psi (t,x)) \theta^{(R)}(x) \theta_{(\varepsilon)} (x) u(t,x) = \theta \parenthese{\abs{x/R}^2 + \varphi (t)} \theta^{(R)} \theta_{(\varepsilon)} u.
\end{align}
We get the following properties of $g$ and of its support.
\begin{itemize}
\item 
From the supports of  $\theta^{(R)}$ and $\theta_{(\varepsilon)}$, 
\begin{equation} \label{eqn-221005-1217-1}
    \supp (g) \subset [0,1] \times (B_R \setminus B_{\varepsilon}).
\end{equation}
\item
The support of the function $\theta (\psi)$ implies that
\begin{align}
\supp (g) \subset \{ \psi (t,x) \geq 1\}.
\end{align}
In particular, when $t \in [0, 1/8] \cup [7/8 , 1]$, $\varphi = 0$, which gives  $\psi = \abs{x/R}^2 \leq 1$ in $( [0, 1/8] \cup [7/8 , 1] ) \times B_R$. Hence, combined with \eqref{eqn-221005-1217-1},
\begin{equation*}
    g =0 \quad \text{for} \,\, t \in [0, 1/8] \cup [7/8 ,1].
\end{equation*} 
In fact, $\theta (\psi)$ is a cutoff in time such that 
\begin{equation*}
    \supp (\theta' (\psi)) \subset \{ 1 \leq \psi \leq 2\},
\end{equation*}
which gives a nice control of time cutoff terms in Carleman estimates. Also, we extend $g$ by zero for $t \notin [0,1]$. This way the extension, that we still denote by $g$, has compact support in $\bR \times \bR^n$. 

\item
When $t \in [1/4, 3/4]$, $\varphi = 3$ and
\begin{align}
\psi = \abs{x/R}^2 + \varphi \geq \abs{x/R}^2  + 3 \geq 3,
\end{align}
which implies $\theta (\psi) =1$ in $[1/4, 3/4]$. This, in particular, gives
\begin{equation}\label{eqn-221005-1223}
g = u \quad \text{ on } \quad [1/4, 3/4] \times (B_{R-1} \setminus B_{2 \varepsilon}).
\end{equation}
\end{itemize}
\medskip

With these properties at hand, we now apply the quadratic Carleman estimate in Lemma \ref{lem-220816-1207} to $f = e^{\phi} g$. 
Below we shall use the notation $LHSC, RHSC$ for the left and the right-hand side of the Carleman estimate in Lemma \ref{lem-220816-1207} applied on $f$, respectively.
Then, for $\beta_1$ defined in \eqref{eqn-221005-0106} with $r_0 = \epsi$, if
\begin{equation}\label{eqn-221005-1224}
    \beta \geq \beta_1,\quad R -1 \geq R_0,
\end{equation}
we have
\begin{align}
LHSC 
& \geq 
\beta^3 R^{-6} \int_{\bR} \int_{\bR^n} \abs{x}^2 \abs{f}^2 \, dxdt 
\left(= 
\beta^3 R^{-6} \int_{\bR} \int_{B_{\epsi}^{\mathsf{c}}} |x|^2 e^{2\phi}|g|^2 \,dxdt \right) \\
& = \frac{1}{2} \beta^3 R^{-6} \int_{\bR} \int_{\bR^n} \abs{x}^2 \abs{f}^2 \, dxdt  + \frac{1}{2} \beta^3 R^{-6} \int_{\bR} \int_{\bR^n} \abs{x}^2 \abs{f}^2 \, dxdt \\
& \geq 
\frac{1}{2} \beta^3 R^{-6} \varepsilon^2 \int_{\bR} \int_{\bR^n} e^{2 \phi} \abs{g}^2 \, dxdt + 2 \beta^3 R^{-6} \varepsilon^2 \int_{ 3/4 }^{ 1/4 } \int_{B_{R_0} \setminus B_{2\varepsilon}} e^{2\beta (4 \frac{\varepsilon^2}{R^2} +3)} \abs{u}^2 \, dxdt \\
& \geq
\frac{1}{2} \beta^3 R^{-6} \varepsilon^2 \int_{\bR} \int_{\bR^n} e^{2 \phi} \abs{g}^2 \, dxdt + 2 \beta^3 R^{-6} \varepsilon^2 e^{2 \beta (4 \frac{\varepsilon^2}{R^2} +3)} E_2^2 \label{eqn-1310-1609}.
\end{align}
Here we used \eqref{eqn-221005-1217-1}, \eqref{eqn-221005-1223}, and \eqref{thm-220815-1155-eq2}.

To estimate the RHSC we first observe that
\begin{align}
(i \partial_t + \cL) g & = (i\partial_t + \cL) (\theta \theta^{(R)} \theta_{(\varepsilon)} u) \\
& = - V g + [i \partial_t + \cL, \theta \theta^{(R)} \theta_{(\varepsilon)}]u,
\end{align}
which splits the RHSC into two terms. For the $Vg$ term, we have
\begin{align}
\int_0^1 \int_{\bR^n} e^{2 \phi} \abs{Vg}^2  \, dxdt\leq \norm{V}_{L^{\infty}}^2 \int_0^1 \int_{\bR^n} e^{2 \phi} \abs{g}^2  \, dxdt.
\end{align}

For the commutator term, we split it into four pieces. For simplicity, in the following we write $\theta$ for $\theta(\psi)$.
\begin{align}
[i \partial_t  + \cL , \theta \theta^{(R)} \theta_{(\varepsilon)}]u & = \underbrace{i \theta' \varphi' \theta^{(R)} \theta_{(\varepsilon)} u + O(1) \frac{\abs{x}}{R^2} \abs{\theta'} (\abs{\nabla A} \theta^{(R)} \theta_{(\varepsilon)} \abs{u} + \abs{\nabla (\theta^{(R)} \theta_{(\varepsilon)} u)}) }_{I_1} \\
& \quad + O(1)  \underbrace{(\frac{1}{R^2} \abs{\theta'} + \frac{\abs{x}^2}{R^4} \abs{\theta''}) \theta^{(R)} \theta_{(\varepsilon)} \abs{u} }_{I_2} \\
& \quad + O(1) \underbrace{ \abs{\nabla \theta_{(\varepsilon)}} \theta (\abs{\nabla A} \theta^{(R)} \abs{u} + \abs{\nabla (\theta^{(R)} u)}) + O(1) \abs{\nabla^2 \theta_{(\varepsilon)}} \theta \theta^{(R)} \abs{u} }_{I_3}\\
& \quad + \underbrace{  \theta \theta_{(\epsi)} [\cL, \theta^{(R)}] u}_{I_4}.\label{eqn-1910-1546}
\end{align}
Note that $I_1$ and $I_2$ are terms with at least one derivative hitting on $\theta (\psi)$, that in $I_3$ no derivative hits on $\theta$ and at least one derivative  hits on $\theta_{(\varepsilon)}$, and that term $I_4$ has derivatives only hitting on $\theta^{(R)}$ but not on $\theta$ nor $\theta_{(\varepsilon)}$.

We first bound $I_1 $ and $I_2$. Since
\begin{align}
\supp (\theta' (\psi)) \subset \{ 1 \leq \psi \leq 2 \} , \implies \phi = \beta \psi \in [\beta, 2 \beta]
\end{align}
and
\begin{align}
\supp (g) \subset \{ t \in [ 1/8 , 7/8 ], \quad \abs{x/R} \leq 1 \} .
\end{align}
Recall \eqref{thm-220815-1155-eq0}, we get
\begin{align}
\int_0^1 \int_{\bR^n} e^{2\phi} \abs{I_1 + I_2}^2 \, dxdt \leq C \int_{ 1/8 }^{ 7/8 } \int_{\bR^n} e^{4 \beta} (\abs{u}^2 + \abs{\nabla u}^2) \, dxdt  \leq C_1 e^{4\beta} E_1^2,
\end{align}
where $C_1 = C_1 (n , \lambda , \Lambda, \norm{\nabla A}_{L_x^{\infty}} , \varepsilon)$.

Next we bound $I_3$. From $\supp (\nabla \theta_{(\varepsilon)}) \subset B_{2\varepsilon} \setminus B_{\varepsilon}$ and $\varphi \leq 3$,
\begin{align}
\int_0^1 \int_{\bR^n} e^{2 \phi} \abs{I_3}^2 \, dxdt  \leq C_2 e^{2 \beta (4 (\frac{\varepsilon}{R})^2 + 3)} \int_{ 1/8 }^{ 7/8 } \int_{B_{2\varepsilon}\setminus B_\varepsilon}(\abs{u}^2 + \abs{\nabla u}^2) \, dxdt  \leq C_2 e^{2 \beta \abs{4 (\frac{\varepsilon}{R})^2 +3}} E_1^2,
\end{align}
where  $C_2 = C_2 (n, \lambda , \Lambda, \norm{\nabla A}_{L_x^{\infty}} , \varepsilon)$.

Finally, for the last term $I_4$, we use that
\begin{align}
e^{\beta (\abs{x/R}^2 + \varphi) } \leq e^{\beta(1 + 3)} \leq e^{4 \beta},\quad \text{in}\,\, [0,1] \times B_R \supset \supp(g),     
\end{align}
which gives 
\begin{align}
\int_0^1 \int_{\bR^n} e^{2\phi} \abs{I_4}^2 \, dxdt  \leq C_3 e^{8 \beta} \int_{ 1/8 }^{ 7/8 } \int_{B_R \setminus B_{R-1}} (\abs{u}^2 + \abs{\nabla u}^2) \, dxdt  = C_3 e^{8 \beta} \delta(R),
\end{align}
with $C_3 = C_3 (n, \lambda , \Lambda, \norm{\nabla A}_{L_x^{\infty}} )$.

Combining all the calculations above, we have
\begin{align}
& \quad \frac{1}{2} \beta^3 R^{-6} \varepsilon^2 \int_0^1 \int_{\bR^n} e^{2\phi} \abs{g}^2 \, dxdt + 2 \beta^3 R^{-6} \varepsilon^2 e^{2 \beta (4 (\frac{\varepsilon}{R})^2 + 3)} E_2^2 \\
& \leq \lambda^{-2} \left( \norm{V}_{L_x^{\infty}}^2 \int_0^1 \int_{\bR^n} e^{2\phi} \abs{g}^2  \, dxdt+  C_1 e^{4 \beta} E_1^2 +  C_2 e^{2 \beta ( 4 (\frac{\varepsilon}{R})^2 +3 ) } E_2^2 +  C_3 e^{8 \beta} \delta (R) \right).
\end{align}

Finally observe that, if besides \eqref{eqn-221005-1224}, we also require that $\beta$ and $R$ satisfy
\begin{equation} \label{eqn-221005-0102}
\beta^3 R^{-6} > 2 \lambda^{-2} \epsi^{-2} \norm{V}_{L_x^{\infty}}^2, 
\quad
\beta^3 R^{-6}\varepsilon^2 E_2^2> \lambda^{-2} ( C_2E_2^2 +  C_1 E_1^2 ),
\end{equation}
then all except the $\delta(R)$ term on the right-hand side can be absorbed, from which
\begin{align}
\delta (R) \geq (\beta^3 R^{-6}\varepsilon^2E^2_2- \lambda^{-2} C_2E_2^2- \lambda^{-2} C_1 E_1^2)\,e^{-2\beta}.
\end{align}

For \eqref{eqn-221005-1224} and \eqref{eqn-221005-0102}, it suffices to choose
\begin{equation*}
    \beta = \frac{ 10 \|\varphi''\|_{L^\infty} }{\lambda \epsi} R^3,
\end{equation*}
which means $C_0 = 10\|\varphi''\|_{L^\infty}/(\lambda\epsi)$, and restrict
\begin{equation*}
    R \geq R_1 := C_4 (n , \lambda , \Lambda,  \norm{A}_{C^3} , M_1, \|\varphi''\|_{L^\infty}, \varepsilon ,  E_1 , E_2)+R_0 +1,
\end{equation*}
with $C_4$ sufficiently large.
This concludes the proof of  Theorem \ref{thm-220815-1155}. 
\end{proof}

\section{Proof of Theorem \ref{thm-220816-1210}}\label{sec Main}

We include the proof of Theorem \ref{thm-220816-1210} in this section.
\begin{proof}
We prove Theorem \ref{thm-220816-1210} by contradiction. Suppose that $u\not\equiv 0$, then there exist some $R_0>0$ sufficiently large and $\epsi>0$ sufficiently small (might both depend on $u$), such that
\begin{equation*}
    \int_{1/4}^{3/4} \int_{B_{R_0} \setminus B_{2\epsi}} |u|^2 \,dxdt \in (0,\infty) .
\end{equation*}
Furthermore, from Theorem \ref{thm-220515-0921},
\begin{equation*}
    \int_{1/8}^{7/8} \int_{\bR^n} (|u|^2 + |\nabla u|^2) \,dxdt < \infty,
\end{equation*}
hence, applying Theorem \ref{thm-220815-1155}, we have the lower bound
\begin{equation}\label{eqn-1309-119}
	\int_{1/8}^{7/8}\int_{B_R\setminus B_{R-1}} |u|^2 + |\nabla u|^2 \,dxdt \geq C_1e^{-C_0R^3},
\end{equation}
with the same $C_0$ chosen as in Theorem \ref{thm-220815-1155}. 
On the other hand, from Corollary \ref{cor-220515-1117} with $\alpha = 3/2$ and the decay properties of $u_0, u_1$, for $\kappa > \kappa_0$, we have
\begin{equation*}
	\int_0^1 \int_{\bR^n} e^{\kappa|x|^3}\big(|u(t,x)|^2 + t(1-t)|\nabla u(t,x)|\big)\,dxdt < \infty,
\end{equation*}
implying that
\begin{equation*}
	\lim_{R\rightarrow \infty} e^{\kappa R^3} \left(\int_{1/8}^{7/8} \int_{B_{R+1}\setminus B_R} |u(t,x)|^2 + |\nabla u(t,x)|^2 \,dxdt \right) = 0.
\end{equation*}
Choosing $\kappa$ large enough, we reach a contradiction with \eqref{eqn-1309-119}. Now Theorem \ref{thm-220815-1155} is proved.
\end{proof}

\section{Proof of Theorem \ref{thm-221005-0218}} \label{sec-221005-0114} 
In this section, we focus on the case when the matrix coefficients $a_{kj}$ are translation-invariant in one spatial variable.  Without loss of generality, we assume such space variable to be $x_1$, and we write a point in $\bR^n$ as
\begin{equation*}
	x = (x_1,x')\in \bR\times\bR^{n-1}.
\end{equation*}
\begin{assumption}\label{ass-220517-1100}
Let
	\begin{equation*}
		A=A(x')=
		\begin{pmatrix}
			\overline{a_{11}} & 0\\0&\widetilde{A}(x')
		\end{pmatrix},
	\end{equation*}
where $\overline{a_{11}}>0$ is a constant, and $\widetilde{A}\in C^3(\bR^{n-1})$ is a symmetric $(n-1)\times(n-1)$ matrix .
\end{assumption}
Later we will impose smallness conditions on $|x'||\nabla_{x'}\widetilde{A}|$. It is worth mentioning that we cannot require the previous decay assumption \eqref{eqn-221104-0447} anymore, since
\begin{equation*}
	\forall x'\in\bR^{n-1}, |x_1||\nabla_{x'}A(x')|\,\,\text{bounded as}\,\,|x_1|\rightarrow\infty \quad \text{implies} \quad \nabla_{x'}A\equiv 0, \quad \text{from which}\,\, A\equiv A(0).
\end{equation*}
Later, in the proof of Theorem \ref{thm-221005-0218}, we will show that if $A$ is as in \eqref{eq Structure}, then it can always be reduced in the form as in Assumption \ref{ass-220517-1100}. For this reason, up to the proof of Theorem \ref{thm-221005-0218} we shall work with $A$ as in Assumption \ref{ass-220517-1100}.

The proof of Theorem \ref{thm-221005-0218} includes the following ingredients:
\begin{itemize}
    \item A $log$-convexity result in Section \ref{sec-221104-0525}; 
    \item A ``better'' Carleman estimate in Section \ref{sec-221104-0526}, from which we derive a ``sharp'' lower bound;
    \item A contradiction argument in Section \ref{sec-221104-0527}.
\end{itemize}

\subsection{$log$-convexity} \label{sec-221104-0525}
The results in Theorem \ref{thm-220515-0921} turn out to be still true under Assumption \ref{ass-220517-1100}.
\begin{theorem}\label{thm-sharp-log}
    Let $u\in C([0,1], L^2(\mathbb{R}^n))$ be a solution to \eqref{eqn-220312-0505} with $\mathcal{L}$ defined as in \eqref{eqn-220517-0913} satisfying Assumption \ref{ass-220517-1100}. Let also $V$ be real-valued and such that $M_1:=\|V\|_{L^\infty}<\infty$.
    	Then there exist a small enough $\epsi_0 = \epsi_0(n,\lambda, \Lambda) >0$ and a large enough $\beta_2 = \beta_2(n,\lambda, \Lambda, \|A\|_{C^3})$, such that if 

    \begin{equation} \label{eqn-221104-1154}
	\sup_{x'\in \bR^{n-1}}|x'||\nabla_{x'}\widetilde{A}|\leq\epsi_0
    \end{equation}
    and
	\begin{equation}\label{thm-sharp-log-eqn-1}
		e^{\beta|x|^2}u(0,x), \, e^{\beta|x|^2}u(1,x) \in L^2(\bR^n), \quad \text{for some} \,\, \beta>\beta_2,
	\end{equation}
then we have 
\begin{equation}\label{thm-sharp-log-eqn-1.1}
		\|e^{\beta|x|^2}u(t,x)\|_{L^2_x}^2 
		\leq
		C e^{ M_1^2 } (\|e^{\beta|x|^2}u(0,x)\|_{L^2_x}^2)^{1-t} (\|e^{\beta|x|^2}u(1,x)\|_{L^2_x}^2)^t,
	\end{equation}
  and
	\begin{equation}\label{thm-sharp-log-eqn-1.2}
	\begin{split}
	&\beta\|\sqrt{t(1-t)}e^{\beta|x|^2}|\nabla u|\|_{L^2_{t,x}}^2 + \beta^3\|\sqrt{t(1-t)}e^{\beta|x|^2}|xu|\|_{L^2_{t,x}}^2 
	\\&\leq 
	 Ce^{M_1^2} ( \|e^{\beta|x|^2}u(0,x)\|^2_{L^2_{x}} + \|e^{\beta|x|^2}u(1,x)\|^2_{L^2_{x}} ).
	\end{split}
	\end{equation}
	Here $C$ is an absolute constant.
	
	In particular, when $A=I_n$, we can take $\beta_2=0$.
	
\end{theorem}

\begin{proof}[Sketch of the proof of Theorem \ref{thm-sharp-log}] Recall that in Section \ref{sec-220517-1106}, the proof of the log-convexity property is based on a lower bound for $[\cS,\cA]$ and on the use of a viscosity argument (see Remark \ref{rmk-220711}). See the detailed strategy in Remark \ref{rmk-1103-1427}. Actually the only place where two different types of assumptions on $a_{jk}$ play a role lies in the the commutator estimate \eqref{eqn-220515-1012}, which we give a proof below.
    
    First notice that, taking the weight function $\phi=\beta |x|^2$, Assumption \ref{ass-220517-1100} implies each summand
    \begin{align}
    a_{kj}(\p_{k}a_{ml})\p_l\phi= 2a_{kj}(\p_{k}a_{ml})x_l=0,\quad \text{if at least one of $j, k, m, l$ equals $1$}.
     \label{eqn-1131024}
    \end{align}
Hence, the formulas in Section \ref{sub-220817}, specifically formulas in \eqref{eqn-1103-1042}, show that 
    \begin{align}
    [\cS,\cA]&=T_2+T_1+T_{0,1}+T_{0,2},
    \end{align}
   where $T_2,T_1,T_{0,1},T_{0,2}$ can be written as
   \begin{align}
   T_2&=-8\beta\int a_{kj}a_{mk}\p^2_{mj}-8\beta
   \sum_{j,k,l,m\geq 2}a_{kj}(\p_k a_{ml})x_l\p^2_{mj}+ 4\beta \sum_{j,k,l,m\geq 2}a_{ml}(\p_m a_{kj})x_l\p^2_{kj},\\
       T_1&=O(1)\beta(|\tA||\nabla_{x'}\tA|+|x'||\nabla_{x'}\tA|^2+|\tA||x'||\nabla_{x'}^2\tA|)\nabla,\\
       T_{0,1}&= 32\beta^3 a_{ml}a_{kj}x_l\,x_j\delta_{km}
        +16\beta^3\sum_{j,m,l\geq 2}a_{ml}(\p_m a_{kj})x_l x_k x_j,\\
        T_{0,2}&= O(1) \beta \big(|A||x'||\nabla_{x'}^3\tA|+|A||\nabla_{x'}^2\tA| +|\nabla_{x'}\tA|^2+|x'||\nabla_{x'}\tA||\nabla_{x'}^2\tA|\big).\label{eqn-1103-1144} 
   \end{align}
   Here and below, the usual Einstein convention of summation (summing from $1$ to $n$) is used if the range is not specified.
 Now, using \eqref{eqn-1103-1144} we can follow the steps in the proof of Lemma \ref{lem-220711} to conclude \eqref{eqn-220515-1012}. Indeed,
    \eqref{eqn-220508-0350-1} - \eqref{eqn-220515-0938}
    become
\begin{align}
		\left<[\cS,\cA]f, f\right>_{L_x^2}
		&= 
		8\beta\int a_{kj}a_{mk}\p_m f\p_j f \,dx+ \sum_{m,l\geq 2} \int \left(8\beta a_{kj}(\p_k a_{ml})x_l\p_m f \p_j f - 4\beta a_{ml}x_l (\p_m a_{kj})\p_k f \p_j f\right) \,dx \\
		&\quad 
		-O(1)\beta\int\big(
		|\tA||\nabla_{x'}\tA|+|x'||\nabla_{x'}\tA|^2+|\tA||x'||\nabla_{x'}^2\tA|
		\big)|\nabla f||f| \,dx \\
		&\quad +
		32\beta^3 \int a_{kj}a_{kl}x_lx_j|f|^2 \,dx + \sum_{m,l\geq 2} \int 16\beta^3  a_{ml}x_l\p_m(a_{kj})x_kx_j|f|^2\,dx  \\
		&\quad - O(1)\beta \int\big(|A||x'||\nabla_{x'}^2\tA|+|A||\nabla_{x'}\tA|+|A||\nabla_{x'}^2\tA| +|\nabla_{x'}\tA|^2\Big)|f|^2 \,dx
		\\&\geq \beta(8\lambda^2 - C|x'||\nabla_{x'}\tA|-C\delta)\int |\nabla f|^2 \,dx + \beta^3(32\lambda^2 - C\epsi_0)\int |xf|^2 \,dx
		\\&\quad - O(1)\frac{1}{\delta}\beta\int\big(
		|\tA||\nabla_{x'}\tA|+|x'||\nabla_{x'}\tA|^2+|\tA||x'||\nabla_{x'}^2\tA|
		\big)|f|^2 \,dx
		\\&\quad - O(1)\beta \int\big(|A||x'||\nabla_{x'}^2\tA|+|A||\nabla_{x'}\tA|+|A||\nabla_{x'}^2\tA| +|\nabla_{x'}\tA|^2\big)|f|^2 \,dx ,\\
		&\geq \beta(8\lambda^2 - C\epsi_0-C\delta) \|\nabla f\|_{L_x^2}^2+\int \beta^3\Big(32\lambda^2 -\frac{C}{\beta^2}( 1+\delta^{-1}) (|A||\nabla_{x'}^2\tA|+
	+|\nabla_{x'}\tA|^2 )\Big)|xf|^2\,dx
	\\
		&\quad -C\beta (1+\delta^{-1})\int \big(|\tA||\nabla_{x'}\tA|+|\nabla^2_{x'}\tA| \big)|f|^2\,dx, \label{eqn-150922-1620}
\end{align}
where the factor $|x_1\nabla A|$ does not appear in the leading terms due to Assumption \ref{ass-220517-1100}. Finally, choosing $\varepsilon_0$ and $\delta$ small enough we get \eqref{eqn-220515-1012}. 

The rest of the proof remains the same with Section \ref{sec-220915-0640}, and hence, is omitted.

\end{proof}

\subsection{``Better'' Carleman estimate and ``sharp'' lower bound} \label{sec-221104-0526}

Next, we prove a Carleman inequality with an ``$x_1$-translated'' weight.

\begin{lemma}\label{lem-220516-1205} 
	Let $\cL$ be the operator in \eqref{eqn-220517-0913} with $a_{kj}$ satisfying Assumption \ref{ass-220517-1100}. Let $R>1$ and $\varphi=\varphi(t) \in C^\infty_c(\bR)$. 
	Then there exists a large constant $c_0= c_0(n, \lambda, \Lambda, \|A\|_{C^3} ,\|\varphi'\|_{L^\infty}, \|\varphi''\|_{L^\infty} )>0$, such that if \eqref{eqn-221104-1154} is satisfied, then for any $R \geq 1$, $f\in C^\infty_c(\bR\times\bR^n)$ with
	\begin{equation}\label{eqn-220513-1138}
		\operatorname{supp}(f)\subset \{|x/R + \varphi(t)\vec{e}_1|\geq 1\},
	\end{equation}
	and 
	\begin{equation} \label{eqn-221012-0209}
	    \beta \geq \beta_3:=c_0 R^2,
	\end{equation}
	we have
	\begin{equation}
	    \frac{\beta}{R^2}\|\nabla_x f\|_{L^2_{t,x}}^2 + \frac{\beta^3}{R^6}\||x/R+\varphi\vec{e}_1|f\|_{L^2_{t,x}}^2
		\leq
		C\|e^{\beta|x/R + \varphi\vec{e}_1|^2}(i\p_t + \cL)e^{-\beta|x/R+\varphi\vec{e}_1|^2}f\|^2_{L^2_{t,x}}.
	\end{equation}
	Here $\vec{e}_1$ is the unit vector $(1,0,\ldots,0)$.
\end{lemma} 

\begin{proof}
    Denote
    \begin{equation*}
        \phi = \phi(t,x) = \beta |x/R + \varphi(t) \vec{e}_1|^2.
    \end{equation*}
For all $j,k,l\in\{1,\ldots,n\}$, Assumption \ref{ass-220517-1100} and the choice of $\phi$ imply
\begin{align} \label{eqn-1103-1750}
    a_{j1}=a_{1j}=a_{11}\delta_{j1}, \quad \p_ka_{l1}&=\p_ka_{1l}=0, \quad \p_1 a_{kl}=0,
\end{align}
and
\begin{align}        
\p^3_{jkl}\phi=0, \quad \p^2_{jk}\phi=2\beta R^{-2}\delta_{jk}, \quad \p^2_{jt}\phi=2\beta R^{-1}\varphi'\delta_{j1}.
\end{align}
Hence,
    \begin{align}
        \p_l a_{jk}&=0 \quad \text{if at least one of $j, k, l$ equals $1$},\\
        a_{ml} (\p_m a_{kj})&=0 \quad  \text{if at least one of $j, k, m, l$ equals $1$},\\
        a_{ml} (\p_m a_{kj})(\frac{x_l}{R} +  \varphi(t) \delta_{l1})&=0 \quad \text{if at least one of $j, k, m, l$ equals $1$}.
        \label{eqn-1103-1749}
    \end{align}
    In particular, there is no $\p_{x_1}$ or $(x_1\cdot)$ in every term of the form $a_{ml} (\p_m a_{kj})(\frac{x_l}{R} +  \varphi(t) \delta_{l1})$.
    Now, formulas \eqref{eqn-1103-1042}  in Section \ref{sub-220817} and \eqref{eqn-1103-1750} -- \eqref{eqn-1103-1749}  lead to
    \begin{align}
        [\cS, \cA] = T_2 + T_1 + T_{0,1} + T_{0,2},
    \end{align}
    where
    \begin{align}
        T_2
        &= 
        - \frac{8\beta}{R^2} a_{kj} a_{mk} \p^2_{mj} - \sum_{j,k,m,l\geq 2} \frac{8\beta}{R} a_{kj} (\p_k a_{ml}) (\frac{x_l}{R} +  \cancel{ \varphi(t) \delta_{l1} } ) \p^2_{mj} + \sum_{j,k,m,l\geq 2} \frac{4\beta}{R} a_{ml} (\p_m a_{kj}) (\frac{x_l}{R} + \cancel{ \varphi(t)\delta_{l1} })\p^2_{kj}
        \\ T_1 &=
        -8i\frac{\beta}{R}\varphi' a_{11}\p_1 + O(1) \frac{\beta}{R^2}\left( |\nabla_{x'} \tilde{A}| +  |x'| |\nabla_{x'} \tA|^2 + |x'| |\nabla^2_{x'}\tA|) \right)\nabla
        \\ T_{0,1} &=
        \frac{32 \beta^3}{R^4} a_{kl} a_{kj} ( \frac{x_l}{R} + \varphi(t) \delta_{l1} ) ( \frac{x_j}{R} + \varphi(t) \delta_{j1} ) 
        \\&\quad
        + \sum_{j,k,m,l \geq 2} \frac{16 \beta^3}{R^3} a_{ml} (\p_m a_{kj}) ( \frac{x_l}{R} + \cancel{\varphi(t) \delta_{l1}} ) ( \frac{x_k}{R} + \cancel{\varphi(t) \delta_{k1}} ) ( \frac{x_j}{R} + \cancel{\varphi(t) \delta_{j1}} )\\
        &=\frac{32 \beta^3}{R^4} a_{kl} a_{kj} ( \frac{x_l}{R} + \varphi(t) \delta_{l1} ) ( \frac{x_j}{R} + \varphi(t) \delta_{j1} ) + O(1)\frac{\beta^3}{R^6}( |\nabla_{x'}\tA||x'|^3 )
        \\ T_{0,2} &=
        2\frac{\beta}{R} (\frac{x_1}{R} + \varphi) \varphi'' + 2\frac{\beta}{R} (\varphi')^2 \\
        &\quad + O(1) \frac{\beta}{R} \big( \left|\frac{x'}{R}\right||\nabla_{x'}^3\tA| + |\nabla_{x'}^2\tA| + |\nabla_{x'}\tA|^2 + \left|\frac{x'}{R}\right||\nabla_{x'}\tA||\nabla_{x'}^2\tA|\big).
\end{align}
To bound the commutator $[\cS,\cA]$ we proceed as in \eqref{eqn-150922-1620} with suitable adjustments. Since 
\begin{equation}
    R \geq 1, \quad |x'/R| \leq |x/R+\varphi(t)e_1| \,\, \forall (x,t), \quad \text{and} \quad  |x/R+\varphi(t)e_1|\geq 1 \,\,  \text{on} \,\,\mathrm{supp}f,
\end{equation} 
what we get is
\begin{align}
    &\left< [\cS, \cA] f , f \right>_{L^2_{t,x}} \\
    &\geq \beta R^{-2}(8\lambda^2-C\varepsilon_0-C  \delta )\|\nabla f\|^2_{L^2_{t, x}}
     - \frac{C}{\delta}\beta  ( \|\varphi'\|_{L^\infty}^2 + \|\tilde{A}\|_{C^2}^2 ) \| |x/R+\varphi(t)e_1| f\|_{L^2_{t,x}}^2 
    \\
    &\quad + \frac{\beta^3}{R^4}(32\lambda^2-C\varepsilon_0 ) \| |x/R+\varphi(t)e_1|f\|^2_{L^2_{t,x}}-\frac{\beta}{R}C(\|\tilde A\|_{C^3},\|\varphi'\|_{L^\infty},\|\varphi''\|_{L^\infty})\| |x/R+\varphi(t)e_1|f\|^2_{L^2_{t,x}}.
\end{align}
Hence, by choosing $\varepsilon_0$ and  $\delta$ small enough,
\begin{align}
&\left< [\cS, \cA] f , f \right>_{L^2_{t,x}}  \\
&\quad \geq 7\beta R^{-2}\|\nabla f\|^2_{L^2_{tx}}+  \beta^3 R^{-4}\big( 31 \lambda^2-C(\|\tilde A\|_{C^3},\|\varphi'\|_{L^\infty},\|\varphi''\|_{L^\infty})(\beta^{-2} R^4 + \beta^{-2}R^3) \big)\| |x/R+\varphi(t)e_1|f\|^2_{L^2_{t,x}}.
\end{align}
Finally, as in the proof of Lemma \ref{lem-220816-1207}, to have the conclusion it suffices to assume $\beta\geq c_0R^2$ with a sufficiently large $c_0 = c_0 (n,\lambda,\Lambda, \|\tilde A\|_{C^3}, \|\varphi'\|, \|\varphi''\|)$, and  with $R\geq 1$. This completes the proof.
\end{proof}
Compared to Lemma \ref{lem-220816-1207}, the ``better'' weight in Lemma \ref{lem-220516-1205} leads to the following sharp lower bound.

\begin{theorem}\label{thm-220513-1127}
 Let $u\in L^\infty L^2 \cap L^2 H^1$ be a solution of \eqref{eqn-220312-0505} with $a_{kj}$ satisfying Assumption \ref{ass-220517-1100} and $V \in L^\infty([0,1]\times\mathbb{R}^n, \mathbb{R})$, and denote by $M_1:=\|V\|_{L^\infty}$.
 
    Furthermore, let $E_1, E_2, R_0$ be the numbers such that
    \begin{equation}\label{thm-150922-1644-eq0}
        \int_{1/8}^{7/8} \int_{\mathbb{R}^n} \left(|u|^2+|\nabla u|^2\right)(t,x) \,dxdt\leq E_1^2<\infty
    \end{equation}
    and
    \begin{equation} \label{thm-150922-1644-eq2}
        \int_{1/4}^{3/4}\int_{B_{R_0}}|u|^2(t,x) \,dxdt \geq E_2^2 > 0,
      \end{equation}
    then there exist some positive constants $R_1=R_1(n, \lambda , \Lambda, \norm{\tilde{A}}_{C^3} , M_1,E_1 , E_2 , R_0)$,  $C_0 = C_0(n, \lambda, \Lambda, \|\tilde{A}\|_{C^3} )$, and $C_1=C_1(n, \lambda , \Lambda, \norm{\tilde{A}}_{C^3} ,M_1,E_1 , E_2 , R_0)$, such that, for any $R > R_1$,
    \begin{equation}\label{thm-150922-1644-eq1}
        \delta(R) := \int_{1/8}^{7/8} \int_{B_R\setminus B_{R-1}} |u|^2 + |\nabla u|^2 \,dxdt \geq C_1\,e^{-C_0R^2}.
    \end{equation}
\end{theorem}

\begin{proof}[Proof of Theorem \ref{thm-220513-1127}]
We define the cutoff functions $\varphi=\varphi(t)$ and $\theta, \theta^{(R)} = \theta, \theta^{(R)} (x)$ as in the proof of Theorem \ref{thm-220815-1155}, and
\begin{equation*}
	g(t,x):=\theta^{(R)}(x)\theta(x/R+\varphi(t)\vec{e}_1)u(t,x).
\end{equation*}
Note that $\theta_{(\epsi)}$ is not needed here.
Clearly, from $\theta^{(R)}$ and $\theta$,
\begin{equation*}
	\operatorname{supp}(g)\subset \{|x|\leq R\}\cap\{|x/R + \varphi(t)\vec{e}_1|\geq 1\}.
\end{equation*}
Furthermore, when $t\in [0,1/8]\cup[7/8,1]$, $\varphi=0$. Hence, for any $|x|\leq R$ and every $t\in [0,1/8]\cup[7/8,1]$,
\begin{equation*}
	|x/R + \varphi(t)\vec{e}_1|= |x/R|\leq 1 \implies \theta(x/R + \varphi(t)\vec{e}_1)=0.
\end{equation*}
From all the above,
\begin{equation}\label{eqn-220523-0134}
	\operatorname{supp}(g)\subset \big([1/8 , 7/8] \times B_R \big) \cap \{|x/R + \varphi(t)\vec{e}_1|\geq 1\}.
\end{equation}
Also, the previous considerations allow us to apply Lemma \ref{lem-220516-1205} to $e^{\beta|x/R+\varphi\vec{e}_1|^2} g(t,x)$ to obtain: for all $\beta\geq \beta_3$, with $\beta_3$ as in \eqref{eqn-221012-0209},
\begin{equation}\label{eqn-220516-1216}
		\frac{\beta}{R^2}\|\nabla_x e^{\beta|x/R+\varphi\vec{e}_1|^2} g\|_{L^2_{t,x}}^2 + \frac{\beta^3}{R^6}\|e^{\beta|x/R+\varphi\vec{e}_1|^2}|x/R+\varphi\vec{e}_1| g\|_{L^2_{t,x}}^2
		\leq
		C\|e^{\beta|x/R + \varphi\vec{e}_1|^2}(i\p_t + \cL) g\|^2_{L^2_{t,x}}.
\end{equation}
Next, we bound the left and right-hand side of the Carleman estimate \eqref{eqn-220516-1216}, denoted, for short, by LHSC and RHSC respectively.

We start with LHSC, which we bound from below as in the proof of Theorem \ref{thm-220815-1155}. Suppose $R \geq R_0 + 1$. For any point
\begin{align}
(t,x)\in [1/4 , 3/4]\times B_{R_0} (\subset [1/4 , 3/4]\times B_{R-1}),
\end{align}
from the construction, we have 
\begin{equation*}
\theta^{(R)}(x)=1\quad \text{and}\quad	\varphi (t,x) = 3,\implies |x/R+\varphi(t)\vec{e}_1|\geq 3 - |x/R|\geq 2 \implies \theta(x/R+\varphi(t)\vec{e}_1)=1,
\end{equation*}
from which
\begin{equation*}
	g(t,x)=u(t,x) \quad \forall (t,x)\in [1/4 , 3/4]\times B_{R_0}.
\end{equation*}
Then, by \eqref{thm-150922-1644-eq2} and the properties of $\mathrm{supp}(g)$, we get
\begin{align}
    LHSC\geq \frac 1 2  \beta^3 R^{-6}\| |x/R+\varphi\vec{e}_1|e^{|x/R+\varphi\vec{e}_1|^2}g\|_{L^2_{t,x}}^2+\frac 1 2 \beta^3 R^{-6} E_2^2\geq \frac 1 2  \beta^3 R^{-6}\|e^{|x/R+\varphi\vec{e}_1|^2}g\|_{L^2_{t,x}}^2+\frac 1 2 \beta^3 R^{-6} E_2^2.
\end{align}

To estimate from above the term RHSC we follow the proof of Theorem \ref{thm-220815-1155}. By repeating the same steps in that proof, 
we get
\begin{align}
    RHSC\leq \|V\|^2_{L^\infty} \|e^{|x/R+\varphi\vec{e}_1|^2}g\|_{L^2_{t,x}}^2 + C\int_0^1\int_{\mathbb{R}^n}(|I_1|^2+|I_2|^2+|I_3|^2+|I_4|^2)\,dxdt,
\end{align}
where $I_j$, $j=1,\ldots,4,$ are as in \eqref{eqn-1910-1546} with $\theta_{(\varepsilon)}$ simply replaced by $1$ (recall, we do not have the cutoff $\theta_{(\varepsilon)}$ in $g$ here). This gives, in particular, that $I_3\equiv 0$. Now,  since $\theta'(x/R+\varphi\vec{e}_1)\neq 0$ if and only if $1<|x/R+\varphi\vec{e}_1|< 2$ and $t\in (1/8,7/8)$, we get
\begin{align}
 \int_0^1\int_{\mathbb{R}^n}(|I_1|^2+|I_2|^2)\,dxdt   \leq C e^{4\beta} E_1^2.
\end{align}
Also, since $|x/R+\varphi\vec{e}_1|^2\leq 16$ on $\mathrm{supp}(g)$, we have
\begin{align}
 \int_0^1\int_{\mathbb{R}^n}|I_4|^2\,dxdt   \leq C e^{32\beta} \delta(R).
\end{align}

Putting together the estimates for LHSC and RHSC, and arguing as in the end of the proof  of Theorem \ref{thm-220815-1155}, that is  by choosing $\beta=C_0 R^2$, with $C_0\geq c_0$ ($c_0$ as in Lemma  \ref{lem-220516-1205}, so that $\beta\geq \beta_3$), and $R$ sufficiently large, we find
$$\delta(R)\geq C_1 e^{-28C_0 R^2}.$$
Note finally that the power $R^2$ (instead of $R^3$ of the general case) in the lower bound for $\delta(R)$ is allowed by the choice $\beta=C_0 R^2$.
\end{proof}

\subsection{Proof of uniqueness} \label{sec-221104-0527}
\begin{proof}[Proof of Theorem \ref{thm-221005-0218}]
	Let us first make some reduction. Take the change of variables
	\begin{equation*}
		\begin{cases}
			y_1 = y_1(x_1)\\
			y'=x'
		\end{cases}\quad\text{with}\quad \frac{dy_1}{dx_1} = \frac{1}{\sqrt{a_{11}(x_1)}}
	\end{equation*}
and the gauge transform
\begin{equation*}
	v(t,y)=e^{\psi}u(t,x(y)),
\end{equation*}
where
\begin{equation*}
	\psi = \psi(y_1) = -\frac{1}{4}\int_0^{y_1}\frac{a_{11}'(s)}{a_{11}(s)}\,ds.
\end{equation*}
One could check that $v$ satisfies
\begin{equation*}
	\left(\p_t + i\p_{y_k}(b_{kj}\p_{y_j}\cdot)\right)v = \left(V + (\p_{y_1}\psi)^2 + \p^2_{y_1y_1}\psi \right)v,
\end{equation*}
where
\begin{equation*}
	B = (b_{kj})_{n\times n} = \begin{pmatrix}
		1&\\&\widetilde{A}(y')
	\end{pmatrix}
\end{equation*}
verifies Assumption \ref{ass-220517-1100}, and
\begin{equation*}
	V(x(y)) - (\p_{y_1}\psi)^2 - \p^2_{y_1y_1}\psi \in L^\infty_{y}.
\end{equation*}
See \cite{MR4332459}. Since
\begin{equation*}
	e^{\psi(y_1)}\leq e^{\frac{1}{4}\|a_{11}'\|_{L^\infty} (\inf\{a_{11}\})^{-1}|y_1|},
\end{equation*}
we still have
\begin{equation}\label{eqn-220524-0709}
	e^{C|y|^2}|u(0,y)|,\,\, e^{C|y|^2}|u(1,y)|\in L^2_y(\bR^n), \,\,\text{for some large $C$ depending only on} \,\, \beta \,\,\text{and} \,\, a_{11}.
	\end{equation}
	The rest of the proof stays the same as in Theorem \ref{thm-220816-1210}, with an application of Theorem \ref{thm-220513-1127} in place of that of Theorem \ref{thm-220815-1155}.
Theorem  \ref{thm-221005-0218} is proved.
\end{proof}

\appendix
\section{A subordination inequality}\label{sec AppA}
\textbf{Notations}. We shall use the notation $f\approx g$ to indicate that there exists a positive constant $C$ such that $1/Cg(x)\leq f(x)\leq Cg(x)$ for every $x$ in a suitable set. When the aforementioned constant depends on some parameters, say $a,b$, we shall often just use the symbol $\approx_{a,b}$. Also  we say $f (x)\lesssim g(x)$, if $f(x) \leq C g(x)$. 

\begin{lemma}\label{lem1.Appendix}
Let $p\in(1,2)$, $\lambda_0>0$, and $q$ the conjugate exponent of $p$, that is such that $1/p+1/q=1$. Then, for any $\kappa>2\lambda_0\left( \frac{2}{q-2}\right)^{\frac{1}{q}}$ and $r>0$,
\begin{equation*}
    \int_{\lambda_0}^{+\infty} e^{\lambda r-\frac{\lambda^q}{q\kappa^q}}\lambda^{\frac{q-2}{2}} \, d\lambda\approx_{\kappa,p}e^{\frac{\kappa^p r^p}{p}}.
\end{equation*}
\end{lemma}
\begin{proof}[Proof of Lemma \ref{lem1.Appendix}]
By Proposition 1 in \cite{EKPV_JLMS} (see also the comments after Proposition 1 in \cite{EKPV_JLMS}), for $p\in (1,2)$, $q$ conjugate exponent of $p$, and $r>0$,   \begin{equation}
    \int_{\mathbb{R}} \underbrace{e^{\lambda r-\frac{|\lambda|^q}{q}}|\lambda|^{\frac{q-2}{2}}}_{I_q(r,\lambda)} \, d\lambda\approx_p e^{\frac{ r^p}{p}},
\end{equation}
where the constant realizing the equivalence $\approx_p$ is  $C=C(p)>1$.
Now, since 
\begin{align}
\int_{\mathbb{R}_-}I_q(r,\lambda) \, d\lambda\leq \int_{\mathbb{R}_+}I_q(r,\lambda) \, d\lambda  ,
\end{align}
we have
\begin{equation*}
   \int_{\mathbb{R}_+}I_q(r,\lambda) \, d\lambda \approx_p  e^{\frac{ r^p}{p}}.
\end{equation*}
Consequently, for any given $\kappa>0$, the change of variable $\lambda'= \lambda\kappa$ yields
\begin{equation*}
   e^{\frac{\kappa^p r^p}{p}}\approx_p \int_{\mathbb{R}_+}I_q(\kappa r,\lambda) \, d\lambda = \kappa^{-1} \int_{\mathbb{R}_+} I_q(\kappa r,\lambda'/\kappa) \, d\lambda'=\kappa^{-q/2}\int_{\mathbb{R}_+} e^{\lambda' r-\frac{\lambda'^q}{q\kappa^q}}\lambda'^{\frac{q-2}{2}} \, d\lambda',
\end{equation*}
that is
\begin{equation*}
    \int_{0}^{+\infty} e^{\lambda r-\frac{\lambda^q}{q\kappa^q}}\lambda^{\frac{q-2}{2}} \, d\lambda\approx_{\kappa,p}e^{\frac{\kappa^p r^p}{p}},
\end{equation*}
where the constant realizing the equivalence is $C(p)\kappa^{q/2}$ with $C(p)>1$.

To conclude the proof we just have to show that for any fixed $\lambda_0>0$
\begin{equation}\label{pf.lem1.Appendix.eq1}
    \int_{0}^{+\infty} e^{\lambda r-\frac{\lambda^q}{q\kappa^q}}\lambda^{\frac{q-2}{2}} \, d\lambda\approx
    \int_{\lambda_0}^{+\infty} e^{\lambda r-\frac{\lambda^q}{q\kappa^q}}\lambda^{\frac{q-2}{2}} \, d\lambda.
    \end{equation}
    Since the relation $LHS$ of \eqref{pf.lem1.Appendix.eq1} $\geq RHS$ of \eqref{pf.lem1.Appendix.eq1} is trivial, we are then left to show
\begin{equation}\label{pf.lem1.Appendix.eq2}
    LHS \text{ of \eqref{pf.lem1.Appendix.eq1}} \lesssim RHS  \text{ of \eqref{pf.lem1.Appendix.eq1}}
    \end{equation}
We now denote by 
\begin{align}
\tilde{I}_{q,r,\kappa}(\lambda):= e^{\lambda r-\frac{\lambda^q}{q\kappa^q}}\lambda^{\frac{q-2}{2}}.
\end{align}
From
\begin{equation*}
    \frac{d}{d\lambda}\tilde{I}_{q,r,\kappa} = ( r - \frac{\lambda^{q-1}}{\kappa^q} + \frac{q-2}{2}\lambda^{-1} ) \tilde{I}_{q,r,\kappa},
\end{equation*}
it is not hard to see that for every fixed $\lambda_0>0$ and every $\kappa \geq 2\lambda_0 \left(\frac{2}{q-2}\right)^{1/q}$ (recall that $p\in (1,2)$, therefore $q>2$), the function $\tilde{I}_{q,r,\kappa}$ (as a function of $\lambda$) is increasing in $(0,2\lambda_0)$. Hence,
\begin{align*}
    LHS \text{ of \eqref{pf.lem1.Appendix.eq1}}&= \int_{0}^{\lambda_0} \tilde{I}_{q,r,\kappa}(\lambda)  \, d\lambda+\int_{\lambda_0}^\infty \tilde{I}_{q,r,\kappa}(\lambda)  \, d\lambda\\
    &\leq \int_{\lambda_0}^{2\lambda_0} \tilde{I}_{q,r,\kappa}(\lambda)  \, d\lambda+\int_{\lambda_0}^\infty \tilde{I}_{q,r,\kappa}(\lambda)  \, d\lambda\\
    &\leq 2 \, RHS  \text{ of \eqref{pf.lem1.Appendix.eq1}},
\end{align*}
which proves \eqref{pf.lem1.Appendix.eq2} and concludes the proof of Lemma \ref{lem1.Appendix}. 
    \end{proof}
\section{Proof of Lemma \ref{lem-220515-0957}}\label{sec AppB}

\begin{proof}[Proof of Lemma \ref{lem-220515-0957}]
	Let $g(x)=(4r^2-|x|^2)f(x)$. Then,
	\begin{equation*}
	g \in H^1(B_{2r}),\quad g=0\,\,\text{on}\,\,\p B_{2r}.
	\end{equation*}
	Applying the Poincar\'e inequality to $g$ on $B_{2r}$, we obtain
	\begin{align*}
	\int_{B_{2r}} |g|^2 \, dx
	&\leq
	C r^2 \int_{B_{2r}} |\nabla g|^2 \, dx
	=
	C r^2 \int_{B_{2r}} |\nabla((4r^2-|x|^2)f)|^2 \, dx
	\\&=
	4C r^2 \int_{B_{2r}} |xf|^2  \, dx + C r^2 \int_{B_{2r}} (4r^2 - |x|^2)^2 |\nabla f|^2 \, dx
	\leq
	4C r^2 \int_{B_{2r}} |xf|^2  \, dx + 16C r^6 \int_{B_{2r}} |\nabla f|^2 \, dx.
	\end{align*}
	Since $4r^2 - |x|^2\geq 3r^2$ for $x\in B_r$,
	\begin{equation*}
	\int_{B_{2r}} |g|^2 \, dx = \int_{B_{2r}} |(4r^2-|x|^2)f(x)|^2 \, dx \geq 9 r^4 \int_{B_r} |f|^2 \, dx.
	\end{equation*}
	Combining the estimates above, Lemma \ref{lem-220515-0957} is proved.
\end{proof}

\section{Proof of Lemma \ref{lem-220706-1211}}\label{sec AppC}

\begin{proof}[Proof of Lemma \ref{lem-220706-1211}]
	From the semigroup property (cf. \cite{MR710486}), we rewrite
	\begin{equation*}
		u (t) = e^{t (a + ib) (\cL + V)} u_0 = e^{a \delta/2 (\cL + V)} e^{(a \delta/2 + ib \delta) (\cL + V)} e^{(t - \delta) (a + ib) (\cL + V)} u_0,
	\end{equation*}
	where $\delta \in (0,t)$ is a small enough number to be chosen later. From \eqref{eqn-220706-1234}, we have 
	\begin{equation*}
		e^{\beta |x|^2} e^{(t - \delta) (a + ib) (\cL + V)} u_0 \in L^2_x.
	\end{equation*}
	Now we apply Lemma \ref{lem-220705-1134} with $\gamma = \beta$, $u(0)$ being replaced by $e^{(t - \delta) (a + ib) (\cL + V)} u_0$, and $g = 0$. By choosing $\delta >0$ small enough according to $\epsi$, we have
	\begin{equation*}
		\| e^{(\beta - \epsi/2) |x|^2} (e^{a s (\cL + V)} e^{(a \delta/2 + ib \delta) (\cL + V)}) e^{(t - \delta) (a + ib) (\cL + V)} u_0 \|_{L^2_x} \leq C \quad \forall s\in (0 , \delta).
	\end{equation*}
	Now the conclusion follows from the standard parabolic regularity theory for the operator
	\begin{equation*}
		\p_t - a (\cL + V).
	\end{equation*}
\end{proof}

\bibliographystyle{plain}
\bibliography{biblio}

\end{document}